\documentclass[12pt]{elsarticle}

\usepackage{amssymb}   
\usepackage{amsthm}    
\usepackage{amsmath}   
\usepackage{stmaryrd}  
\usepackage{titletoc}  
\usepackage{mathrsfs}  
\usepackage{cases}

\theoremstyle{plain}

\newtheorem{thm}{Theorem}[section]
\newtheorem{cor}[thm]{Corollary}
\newtheorem{lem}[thm]{Lemma}
\newtheorem{prop}[thm]{Proposition}

\newtheorem{defn}{Definition}[section]

\newtheorem{rmk}{Remark}[section]

\newcommand{\R}{\mathbb{R}}
\newcommand{\bE}{\mathbb{E}}
\newcommand{\bR}{\mathbb{R}}

\newcommand{\bP}{\mathbb{P}}

\newcommand{\cL}{\mathcal{L}}
\newcommand{\cM}{\mathcal{M}}

\newcommand{\sF}{\mathscr{F}}
\newcommand{\sP}{\mathscr{P}}

\newcommand{\A}{\textbf{A}}

\newcommand{\PC}{\textbf{P}}
\newcommand{\F}{\textbf{F}}

\newcommand{\eps}{\varepsilon}
\newcommand{\vf}{\varphi}

\newcommand{\la}{\langle}
\newcommand{\ra}{\rangle}

\newcommand{\ptl}{\partial}
\newcommand{\rrow}{\rightarrow}

\newcommand{\wt}{\widetilde}
\newcommand{\wh}{\widehat}

\makeatletter
   
   \@addtoreset{equation}{section}
\makeatother

\begin{document}

\begin{frontmatter}

\title{Semi-linear Degenerate Backward Stochastic Partial Differential
Equations and Associated Forward Backward Stochastic Differential
Equations\tnoteref{label1}}

\tnotetext[label1]{This paper is supported by the National Basic Research
Program of China (973 Program) with grant No. 2007CB814904, National Natural
Science Foundation of China with grant No. 10771122 and 11101090, the
Doctoral Program of Higher Education of China with grant No. 20090071110001
and the Specialized Research Fund for the Doctoral Program of Higher
Education of China with grant No. 20090071120002.}

\author[label2]{Kai Du}
\ead{kdu@fudan.edu.cn}

\author[label2]{Qi Zhang}
\ead{qzh@fudan.edu.cn}

\address[label2]{Department of Finance and Control Sciences, School of
Mathematical Sciences, Fudan University, Shanghai 200433, China.}

\begin{abstract}
In this paper, we consider the Cauchy problem of semi-linear degenerate
backward stochastic partial differential equations (BSPDEs in short) under
general settings without technical assumptions on the coefficients. For the
solution of semi-linear degenerate BSPDE, we first give a proof for its
existence and uniqueness, as well as regularity. Then the connection between
semi-linear degenerate BSPDEs and forward backward stochastic differential
equations (FBSDEs in short) is established, which can be regarded as an
extension of Feynman-Kac formula to non-Markov frame.
\end{abstract}


\end{frontmatter}

\section{Introduction}

BSPDEs were introduced by Bensoussan as the adjoint equation of SPDE control
systems. Since then BSPDEs have been applied to control theory and many other
research fields. For example, in the study of stochastic maximum principle
for stochastic parabolic PDEs or stochastic differential equations (SDEs in
short) with partial information, the adjoint equations of
Duncan-Mortensen-Zakai filtering equations are needed to solve, which are
actually BSPDEs. For this kind of application, one can refer to \cite{ha,
na-ni, ta, zh}, to name but a few. Moreover, by means of the classical
duality argument, the controllability of stochastic parabolic equations can
be reduced to the observability estimate for BSPDEs, and this duality
relation was utilized in e.g. \cite{ba-ra-te, ta-zh}. Besides the application
in control theory, BSPDEs are also used in the stochastic process theory and
mathematical finance, and we recommend the reader to see \cite{ch-pa-yo,
en-ka, ma-yo, ma-yo2} for more details.

However, the solvability and the regularity of BSPDE, even for linear BSPDE,
are tough problems due to the differential operators in the form and its
non-Markov characteristic. The recent work \cite{du-ta-zh} by Du, Tang and
Zhang made some progress and lifted the restrictions on the technical
conditions for  the Cauchy problem of linear degenerate BSPDEs. This work
motivates us to consider the Cauchy problem of semi-linear degenerate BSPDEs
under general settings. Actually, non-linear stochastic equations bear more
application backgrounds without the exception of non-linear BSPDEs. For
instance, Peng \cite{pe2} discussed the Bellman dynamic principle for
non-Markov processes, whose corresponding backward stochastic
Hamilton-Jacobi-Bellman equation is a fully non-linear BSPDE. Moreover, in
many subjects of mathematical finance, such as imperfect hedging, portfolio
choice, etc., non-linear BSPDEs appear as an important role and one can
consult \cite{ma-te,mu-za} for this aspect if interested.

Needless to say, more difficulties lay on the solvability of non-linear
BSPDEs. In fact, the solvability of solution to the fully non-linear BSPDE
put forward in \cite{pe2} is still an open problem, under general settings.
Even for semi-linear BSPDE below we consider in this paper, only few work
studied on it:
\begin{numcases}{}\label{eq:main}
  du = - \big[ \cL u + \cM q + f(t,x,u,q+u_{x}\sigma) \big]dt
  + q^{k} dW^{k}_{t}\nonumber\\
u(T,x) = \vf(x),\quad x\in \R^{d},
\end{numcases}
where
\begin{eqnarray*}\label{eq:LM}
    \cL u := a^{ij}u_{x^ix^j} + b^{i}u_{x^i} + cu\ \ {\rm and}\ \ \cM q := \sigma^{ik}q^{k}_{x^i} + \nu^{k}q^{k}.
\end{eqnarray*}
In 2002, Hu, Ma and Yong considered semi-linear BSPDE of above form, under
some specific settings and technical conditions in \cite{hu-ma-yo}. For
instance, they only considered one-dimensional equation and the coefficients
$\sigma,\nu$ were independent of $x$. One of our goal in this paper is to
lift these restrictions and derive the existence, uniqueness and regularity
of semi-linear degenerate BSPDE without technical assumptions. Also we would
like to indicate that the similar regularity of solutions are obtained in
this paper, but much weaker regularity requirements on the coefficients are
needed in comparison with \cite{hu-ma-yo}.

Our another motivation is to establish the correspondence between semi-linear degenerate BSPDE and FBSDE. It is well known that, in Markov frame, the  Feynman-Kac formula for semi-linear equations was established by Peng \cite{pe} and Pardoux-Peng  \cite{pa-pe}. This Feynman-Kac formula demonstrates a correspondence between semi-linear PDE and FBSDE whose coefficients are all Markov processes. But in the non-Markov frame, FBSDE does not correspond to a deterministic PDE any more, but a BSPDE instead, by stochastic calculus. 
Certainly, as an extension of Feynman-Kac formula, this kind of
correspondence is basically important, whether in Mathematical finance
research field or in a potential application to numerical calculus of BSPDE.
To get the correspondence, one necessary step is to derive the continuity of
solution to FBSDE.  Similar to \cite{pa-pe}, we utilize the Kolmogorov
continuity theorem to prove it. But in our settings, no uniform Lipschitz
conditions for $\varphi(x)$ and $f(s,x,0)$ with respect to $x$ are assumed.
Instead we suppose that $\varphi(\cdot)$ and $f(s,\cdot,0)$ belong to
$W^{1,p}$ space and use the Sobolev embedding theorem to get the desired
continuity.

Although \cite{hu-ma-yo, ma-yo} discussed the correspondence between BSPDE
and FBSDE, our conditions are weaker but results are stronger in the solvable
case, and thus can be applied to more equations. We expect that this kind of
correspondence under our settings has independent interest in the areas of
both SPDEs and backward stochastic differential equations (BSDEs in short).

The rest of this paper is organized as follows. In Section 2, we clarify all
necessary notations and state the existing results used in this paper. In
Section 3, we prove the existence, uniqueness and regularity of solution to
semi-linear degenerate BSPDE. The correspondence between semi-linear
degenerate BSPDEs and FBSDEs is established in Section 4.

\section{Preliminaries}
\setcounter{equation}{0}

\ \ \ \ \ Let $(\Omega,\mathscr{F},\{\mathscr{F}_{t}\}_{t\geq 0},\bP)$ be a
complete filtered probability space, among which the filtration
$\{\mathscr{F}_{t}\}_{t\geq 0}$ is generated by a $d'$-dimensional Wiener
process $W=\{W_t;t\geq 0\}$ and all the $\bP$-null sets in $\mathscr{F}$.
Denote by $\mathscr{P}$ the predictable $\sigma$-algebra associated with
$\{\mathscr{F}_{t}\}_{t\geq 0}$.


The following notations will be used in this paper:

$\bullet$ 
For any multi-index $\gamma=(\gamma_1,\dots,\gamma_d)$, we denote
$$D^{\gamma}=D^{\gamma}_x:=\bigg(\frac{\ptl}{\ptl x^1}\bigg)^{\gamma_1}
\bigg(\frac{\ptl}{\ptl x^2}\bigg)^{\gamma_2} \cdots \bigg(\frac{\ptl}{\ptl
x^d}\bigg)^{\gamma_d}$$ and $|\gamma|=\gamma_1+\cdots +\gamma_d$.

$\bullet$ For $n\in\mathbb{Z}^+$, $0<\alpha<1$, denote by
$C_{0}^{\infty}=C_{0}^{\infty}(\mathbb{R}^d)$ the set of infinitely
differentiable real functions of compact support on $\R^d$, by
$C^n=C^n(\mathbb{R}^d)$ the set of $n$ times continuously differentiable
functions on $\R^d$ such that
$$\|u\|_{C^n} := \sum_{|\gamma|\leq n} \sup_{x\in \R^d} |D^{\gamma}u(x)| <
\infty,$$ and by $C^{n,\alpha}=C^{n,\alpha}(\mathbb{R}^d)$ the set of
H$\ddot{\rm o}$lder continuity functions on $\R^d$ such that $$ \left\Vert
u\right\Vert _{C^{n,\alpha }}:=\left\Vert u\right\Vert
_{C^{n}}+\sum_{\left\vert \gamma \right\vert =n}\sup_{x,y\in \mathbb{R}%
^{d},x\neq y}\frac{\left\vert D^{\gamma }u(x)-D^{\gamma }u(y)\right\vert }{%
\left\vert x-y\right\vert ^{\alpha }}<\infty.$$

$\bullet$ For $p>1$ and integer $m\geq 0$, we denote by
$W^{m,p}=W^{m,p}(\bR^d)$ the Sobolev space of real functions on $\bR^d$ with
a finite norm
$$\|u\|_{m,p}:=\bigg(\sum_{|\gamma|\leq
m}\int_{\bR^d}|D^{\gamma}u|^{p}dx\bigg) ^{\frac{1}{p}},$$ where $\gamma$ is a
multi-index. In particular, $W^{0,p}=L^p$. It is well known that $W^{m,2}$ is
a Hilbert space and its inner product is denoted by $\la \cdot,\cdot \ra_m$.

$\bullet$ For $p>1$ and integer $m\geq 0$, we denote by
$W^{m,p}(d')=W^{m,p}(\R^d;\R^{d'})$ the Sobolev space of $d'$ dimensional
vector-valued functions on $\bR^d$ with the norm $\|v\|_{m,p} =
(\sum_{k=1}^{d'}\|v^{k}\|_{m,p}^{p})^{1/p}$.

$\bullet$ Denote by $L^{p}_{\sP}W^{m,p}$ (resp. $L^{p}_{\sP}W^{m,p}(d')$) the
space of all predictable process $u: \Omega\times[0,T]\longrightarrow L^p$
(resp. $u: \Omega\times[0,T]\longrightarrow L^p(\mathbb{R}^{d'})$) such that
$u(\omega,t)\in W^{m,p}$ (resp. $u(\omega,t)\in W^{m,p}(d')$) for a.e.
$(\omega,t)$ and
$$\mathbb{E}\int_0^T\|u(t)\|_{m,p}^p dt <\infty.$$

$\bullet$ Denote by $L^{p}_{\sP}CW^{m,p}$ (resp. $L^{p}_{\sP}C_{w}W^{m,p}$)
the space of all predictable process $u: \Omega\times[0,T]\longrightarrow
W^{m,p}$ strongly (resp. weakly) continuous with respect to $t$ on $[0,T]$
for a.s. $\omega$, such that
$$\mathbb{E}\sup_{t\in[0,T]}\|u(t)\|_{m,p}^p < \infty.$$

Moreover, throughout this paper the summation convention is in force for
repeated indices.

For the coefficients in the semi-linear BSPDE (\ref{eq:main}), we always
assume that $a=(a^{ij})_{d\times d},\ b=(b^{1},\cdots,b^{d}),\ c,\
\sigma=(\sigma^{ik})_{d\times d'}$ and $\nu=(\nu^1,\cdots,\nu^{d'})$ are
$\mathscr{P} \times \mathscr{B}(\R^d)$-measurable with values on the set of
real symmetric $d\times d$ matrices, $\mathbb{R}^{d}$, $\mathbb{R}^1$,
$\mathbb{R}^{d\times d'}$ and $\mathbb{R}^{d'}$, respectively; the real
function $f(t,x,v,r)$ defined on
$\Omega\times[0,T]\times\R^{d}\times\R^1\times\R^{d'}$ is $\mathscr{P} \times
\mathscr{B}(\R^d)$-measurable for each $(v,r)$ and continuous in $(v,r)$ for
each $(\omega,t,x)$; the real function $\vf$ is $\mathscr{F}_T\times
\mathscr{B}(\R^d)$-measurable. Moreover, the following conditions are
needed.\medskip

\noindent \emph{Hypotheses.} For a given constant $K_m\geq 0$ and a given
integer $m \geq 0$,

$(\A_m)$~ the functions $b^i,c,\nu^k$ and their derivatives with respect to
$x$ up to the order $m$, as well as $a^{ij},\sigma^{ik}$ and their
derivatives up to the order $\max\{2,m\}$, are bounded by $K_m$;

$(\PC)$~ (\emph{parabolicity}) for each $(\omega,t,x)\in
\Omega\times[0,T]\times\bR^d$,
$$\big[ 2 a^{ij}(t,x)-\sigma^{ik}\sigma^{jk}(t,x) \big]
\xi^i\xi^j\geq 0,\ {\rm for}\ {\rm arbitrary}\ \xi\in \bR^d.$$

\begin{defn}\label{def:sol}
  We call a pair functions $(u,q)\in L^{2}_{\sP}W^{1,2}\times
  L^{2}_{\sP}W^{0,2}(d')$ a (generalized) solution of BSPDE
  \eqref{eq:main} if for each $\eta\in
  C^{\infty}_{0}$ and a.e. $(\omega,t)$,
  \begin{eqnarray}\label{eq:sol}\begin{split}
    \la u(t),\eta\ra_{0}=&\la\vf,\eta\ra_{0}+\int_{t}^{T}
    \la \cL u + \cM q + f(s,x,u,q + u_{x}\sigma),
    \eta \ra_0 ds\\
    &-\int_{t}^{T}\la q(s),\eta\ra_0 dW_{s}\ \ \bP-{\rm a.s.}
    \end{split}
  \end{eqnarray}
\end{defn}

\begin{rmk}
  In \eqref{eq:sol}, the term $\la a^{ij}u_{x^ix^j},\eta \ra_0$ is
  understood as $$-\la a^{ij}u_{x^i},\eta_{x^j} \ra_0
  - \la a^{ij}_{x^j}u_{x^i},\eta \ra_{0}.$$
\end{rmk}
\smallskip

For convenience, we do a transform in equation \eqref{eq:main} by setting
\begin{equation}\label{eq:wh.q}
  \wh{q} = q + u_{x}\sigma.
\end{equation}
Define $\alpha^{ij} = \frac{1}{2}\sigma^{ik}\sigma^{jk}$. Then equation
\eqref{eq:main} can be rewritten as the following form:
\begin{numcases}{}\label{eq:trans}
  du = - \big[ \widehat{\cL} u + \cM \wh{q} +
  f(t,x,u,\wh{q}) \big]dt
  + (\wh{q}-u_{x}\sigma) dW_{t}\nonumber\\
u(T,x) = \vf(x),\quad x\in \R^{d},
\end{numcases}{}
where
\[\begin{array}{c}
\widehat{\cL}u = (a^{ij}-2\alpha^{ij}) u_{x^i x^j} + \wt{b}^{i} u_{x^i} + c
u\ \ {\rm and}\ \ \wt{b}^i = b^{i} - \sigma^{ik}_{x^j} \sigma^{jk} - \nu^{k}
\sigma^{ik}.
\end{array}\]

It is clear that a function pair $(u,q)$ satisfies \eqref{eq:sol} if and only
if $(u,\hat{q})$ satisfies the following
\begin{eqnarray}\label{eq:sol2}
    \la u(t),\eta\ra_{0}&=&\la\vf,\eta\ra_{0}+\int_{t}^{T}
    \la \widehat{\cL} u(s) + \cM \wh{q}(s) + f(s,x,u(s),\wh{q}(s)),
    \eta \ra_0 ds\nonumber\\
    &&-\int_{t}^{T}\la \wh{q}(s)-u_{x}\sigma(s),
    \eta\ra_0 dW_{s}~.
\end{eqnarray}

To investigate semi-linear BSPDEs, we need some results about linear
equations. In the linear case, $f$ in equation \eqref{eq:main} is taken to be
independent of the last two variables, i.e.
\[f(t,x,v,r) = F(t,x),\]
where the real function $F$ is $\sP\times B(\R^d)$-measurable. Then we have

\begin{thm}\label{lem:l.eq.2} (Theorem 2.1 in \cite{du-ta-zh})
  Let conditions \emph{$(\A_m)$} and \emph{$(\PC)$} be satisfied for given $m \geq 1$.
  If
    $f \in L^{2}_{\sP}W^{m,2}$ and $\vf \in L^2_{\sF_T}(\Omega;W^{m,2})$, then BSPDE \eqref{eq:main} has a unique generalized
  solution $(u,q)$ such that
    $$u \in L^{2}_{\sP}C_{w}W^{m,2}~~{\rm and}~~
    q + \nabla u \,\sigma \in L^{2}_{\sP}W^{m,2}(d'),$$ and
    for any integer $m_1 \in [0,m]$, we have the estimates
    \begin{eqnarray}\label{est:l.eq.2}
      \nonumber &&\mathbb{E} \sup_{t \leq T} \|u(t)\|_{m_1,2}^{2}
      + \mathbb{E}\int_{0}^{T} \|(q + \nabla u\,\sigma)(t)\|_{m_1,2}^{2}\,dt
      \\
      &&\leq~C \mathbb{E}\bigg( \|\vf\|_{m_1,2}^{2}
      + \int_{0}^{T} \|f(t)\|_{m_1,2}^{2}\,dt \bigg),
    \end{eqnarray}
    $C$ is a
   generic constant which depends only on $d,d',K_m,m$ and $T$.

    In addition, if
    $f \in L^{p}_{\sP}W^{m,p}$ and $\vf \in L^p_{\sF_T}(\Omega;W^{m,p})$
    for $p \geq 2$, then $u \in L^{p}_{\sP}C_{w}W^{m,p}$, and
    for any integer $m_1 \in [0,m]$,
    \begin{equation*}\label{est:l.eq.p}
      \mathbb{E} \sup_{t \leq T} \|u(t)\|_{m_1,p}^{p}
      \leq C e^{Cp}\mathbb{E}\bigg( \|\vf\|_{m_1,p}^{p}
      + \int_{0}^{T} \|f(t)\|_{m_1,p}^{p}\,dt \bigg).
    \end{equation*}
\end{thm}

In the remaining of this paper, we still use $C>0$ as a generic constant only
depending on given parameters, and when needed, a bracket will follow
immediately after $C$ to indicate what parameters $C$ depends on.

However, \eqref{est:l.eq.2} is not enough to obtain the estimates of the
solution to the semi-linear equation, and we need more preparations. First
let's see a lemma below.

\begin{lem}\label{lem:lambda.est}
  Let conditions \emph{$(\A_1)$} and \emph{$(\PC)$} be satisfied.
  Then 
  there exists a positive constant $C(d,d',K_1,T)$
  such that for any positive number $\lambda > C+1$,
  \begin{equation}\label{est:lambda}
    \mathbb{E}\int_{0}^{T} e^{\lambda t}\Big( \|u\|_{1,2}^{2}
    + \|q+u_{x}\sigma\|_{1,2}^{2} \Big) dt
    \leq 2 e^{\lambda T}\mathbb{E} \|\vf\|_{1,2}^{2}
    + \frac{2}{\lambda-C-1} \mathbb{E}\int_{0}^{T} e^{\lambda t}\|F(t)\|_{1,2}^{2}dt.
  \end{equation}
\end{lem}
\begin{proof}
Take a small number $\eps > 0$. Consider the following BSPDE with
super-parabolic condition:
\begin{numcases}{}\label{eq:eps}
    d u^{\eps} = -\big[(\eps \Delta + \cL) u^{\eps} + \cM q^{\eps}
    + F\big]dt + q^{\eps} dW_t\nonumber\\
    u^{\eps}(T) = \vf.
\end{numcases}{}
In view of Theorem 2.3 in Du-Meng \cite{du-me}, equation \eqref{eq:eps} has a
unique solution $(u^{\eps},q^{\eps})$ satisfying
  $$u^{\eps} \in L^{2}_{\sP}W^{2,2} \cap L^{2}_{\sP}CW^{1,2},~~~~
  q^{\eps} \in L^{2}_{\sP}W^{1,2}(d').$$
Doing a similar transformation as in (\ref{eq:wh.q}) with $\wh{q}^{\eps} =
q^\eps + u^\eps_x \sigma$ and applying It\^o formula (c.f. \cite{kr-ro}) to
$e^{\lambda t}(\sum_{|\alpha|\leq 1}|D^{\alpha}u^\eps|^2 )$, we have
\begin{eqnarray}\label{proof:001}
    && \mathbb{E} \|u^\eps(0)\|_{1,2}^2 - e^{\lambda T}\mathbb{E} \|\vf\|_{1,2}^2
    + \lambda \mathbb{E}\int_{0}^{T}e^{\lambda t} \|u^\eps(t)\|_{1,2}^2 dt\nonumber\\
    && =
    2\sum_{|\alpha|\leq 1} \mathbb{E} \int_0^T e^{\lambda t} \la D^{\alpha}u^\eps,
    D^{\alpha}\big[(\eps \Delta + \wh{\cL}) u^{\eps} + \cM \wh{q}^{\eps}
    + F\big]\ra_0 dt\nonumber\\
    &&~~~~ - \mathbb{E} \int_0^T e^{\lambda t}
    \|\wh{q}^{\eps}-u^{\eps}_x\sigma\|_{1,2}^2dt.
\end{eqnarray}
From Lemma 3.1 in Du-Tang-Zhang \cite{du-ta-zh}, we know that there exists a
constant $C$ depending only on $d,d',K_1,T$, but not $\eps$, such that
\begin{eqnarray*}
  \begin{split}
    & 2\sum_{|\alpha|\leq 1} \la D^{\alpha}u^\eps,
    D^{\alpha}\big[(\eps \Delta + \wh{\cL}) u^{\eps} + \cM \wh{q}^{\eps}
    + f\big]\ra_0
    - \|\wh{q}^{\eps}-u^{\eps}_x\sigma\|_{1,2}^2\\
    & \leq - {1\over 2}\|\wh{q}^\eps(t)\|_{1,2}^{2}
    + C\|u^{\eps}(t)\|_{1,2}^{2}
    + 2\la u^\eps, F \ra_1.
  \end{split}
\end{eqnarray*}
This along with \eqref{proof:001} yields that
\begin{eqnarray}\label{100}\begin{split}
    &\frac{1}{2}\mathbb{E}\int_{0}^{T}e^{\lambda t} \Big(\|u^\eps(t)\|_{1,2}^2
    + \|\wh{q}^\eps(t)\|_{1,2}^{2} \Big) dt\\
    & \leq e^{\lambda T}\mathbb{E} \|\vf\|_{1,2}^2 + (C-\lambda+1)
    \mathbb{E}\int_{0}^{T}e^{\lambda t}\|u^{\eps}(t)\|_{1,2}^{2}dt
    + 2\mathbb{E} \int_{0}^{T}e^{\lambda t}\la u^\eps, F \ra_1(t)dt.
\end{split}\end{eqnarray}
Then taking $\lambda > C +1$ and noting that
\[2\la u^\eps, F \ra_1(t) \leq (\lambda-C-1)\|u^{\eps}(t)\|_{1,2}^{2}
+\frac{1}{\lambda-C-1}\|F(t)\|_{1,2}^{2},\] we obtain estimate
\eqref{est:lambda} for $(u^\eps,q^\eps)$.\medskip

In view of the proof of Theorem 2.1 in Du-Tang-Zhang \cite{du-ta-zh}, we know
that there exists a subsequence $\{\eps_n\}\downarrow 0$ such that
$(u^\eps,\wh{q}^\eps)$ converges weakly to $(u,\wh{q})$ in
$L^2_{\sP}W^{1,2}\times L^2_{\sP}W^{1,2}(d')$ as $n\rrow \infty$. Hence
estimate \eqref{est:lambda} follows from the resonance theorem and the proof
is complete.
\end{proof}

\begin{rmk}
(i) Following the proof of Lemma \ref{lem:lambda.est}, we can easily prove
\begin{eqnarray}\label{rmk:c01}
    \mathbb{E}\int_{0}^{T} e^{\lambda t}\Big( \|u\|_{0,2}^{2}
    + \|q+u_{x}\sigma\|_{0,2}^{2} \Big) dt
    \leq2 e^{\lambda T}\mathbb{E} \|\vf\|_{0,2}^{2}
    + \frac{2}{\lambda-C-1} \mathbb{E}\int_{0}^{T} e^{\lambda t}\|F(t)\|_{0,2}^{2}dt\ \ \ \ \ \
  \end{eqnarray}
  with the identical constant $C$ in Lemma \ref{lem:lambda.est}.

(ii) If we further assume that \emph{$(\A_2)$} holds,
  $f\in L^{2}_{\sP}W^{2,2}$ and $\vf \in L^{2}_{\sF_T}(\Omega;W^{2,2})$, then
  we can similarly deduce that there exists a positive constant $C(d,d',K_2,T)$
  such that for any positive number $\lambda >C+1$,
  \begin{eqnarray}\label{c4:est:lambda-2}
    &&\bE\int_{0}^{T} e^{\lambda t}\Big( \|u\|_{2,2}^{2}
    + \|q+u_{x}\sigma\|_{2,2}^{2} \Big) \,dt\nonumber\\
    &&\leq 2 e^{\lambda T}\bE \|\vf\|_{2,2}^{2}
    + \frac{2}{\lambda-C-1} \bE\int_{0}^{T}
    e^{\lambda t}\|f(t)\|_{2,2}^{2}\,dt.
\end{eqnarray}
\end{rmk}

\section{Existence, uniqueness and regularity of solutions to semi-linear BSPDEs}
\setcounter{equation}{0}

\ \ \ \ We make a further hypothesis on the function $f$ in BSPDE
\eqref{eq:main}:
\medskip

$(\F)$~ the function $f(t,x,v,r)$ satisfies\\
(1) for arbitrary $(\omega,t,x,v,r)$, $f_x$, $f_v$ and $f_r$ exist;\\
(2) $f(\cdot,\cdot,0,0) \in L^{2}_{\sP}W^{1,2}$;\\
(3) there exists a constant $L>0$ such that for each $(\omega,t,x)$,
\begin{eqnarray*}&&|f(t,x,v_1,r_1)-f(t,x,v_2,r_2)|
  +\|f_x(t,x,v_1,r_1)-f_x(t,x,v_2,r_2)\|\\
  &&\leq L (|v_1 - v_2| + \|r_1 - r_2\|),~~~~~~{\rm for}\ {\rm arbitrary}~v_1,v_2\in \R,
  ~~r_1,r_2\in \R^{d'}.
  \end{eqnarray*}

Obviously, $f_v$ and $f_r$ are bounded by the constant $L$.
\medskip

First we give the proof for the existence and uniqueness of solutions to
semi-linear BSPDEs.

\begin{thm}\label{thm:sl.L2}
  Let conditions \emph{$(\A_1)$}, \emph{$(\PC)$}
  and \emph{$(\F)$} be satisfied. Suppose
  $\vf\in L^2_{\sF_T}(\Omega;W^{1,2})$, then BSPDE
  \eqref{eq:main} has a unique solution
  $(u,q)$ such that
  $$u\in L^{2}_{\sP}C_wW^{1,2},~~q+u_{x}\sigma \in L^{2}_{\sP}W^{1,2}(d').$$
  Moreover, there exists a constant $C(d,d',K_1,T,L)$
  such that
  \begin{eqnarray}\label{est:sl}
    \mathbb{E}\sup_{t\in[0,T]}\|u(t)\|_{1,2}^{2} + \mathbb{E}\int_{0}^{T}\|q +
    u_{x}\sigma\|_{1,2}^{2}(t)dt
    \leq C \mathbb{E} \bigg(\|\vf\|_{1,2}^{2}
    + \int_{0}^{T}\|f(t,\cdot,0,0)\|_{1,2}^{2}dt\bigg).\ \ \ \
    \end{eqnarray}
\end{thm}

\begin{proof}
We mainly use the Picard iteration in the proof of this theorem.

\emph{Step 1.} Define a successive sequence by setting
$$(u_0,q_{0}) = (0,0)$$
and $\{(u_{n},q_{n})\}_{n\geq1}$ to be the unique solution of the following
equations:
\begin{numcases}{}\label{eq:iteration}
    du_{n} = - \big[ \cL u_{n} + \cM q_{n}
  + f(t,x,u_{n-1},q_{n-1}+u_{n-1,x}\sigma) \big]dt + q_{n-1}^{k}dW^{k}_{t}\nonumber\\
    u_{n}(T) = \vf.
\end{numcases}
The solvability of equation \eqref{eq:iteration} is indicated by Theorem
\ref{lem:l.eq.2} since one can easily check that
\[f(\cdot,\cdot,u_{n-1},q_{n-1}+u_{n-1,x}\sigma) \in L^{2}_{\sP}W^{1,2} \]
by virtue of condition $(\F)$. Then we obtain a sequence
$\{(u_{n},\wh{q}_n)\}_{n\geq0}\subset L^{2}_{\sP}C_wW^{1,2} \times
L^{2}_{\sP}W^{1,2}(d')$, where
\[ \wh{q}_{n} = q_{n}+u_{n,x}\sigma. \]

\emph{Step 2.} For the sequence $\{(u_{n},\wh{q}_n)\}_{n\geq0}$ defined in
Step 1, we prove that a subsequence converges weakly in $L^{2}_{\sP}W^{1,2}
\times L^{2}_{\sP}W^{1,2}(d')$. First noticing condition $(\F)$, we have that
for each integer $n\geq 1$, there exists a positive constant $\widetilde{C}$
depending only on $L$ such that
\begin{eqnarray}\label{proof:003}\begin{split}
  &\mathbb{E} \int_0^Te^{\lambda t}\|f(t,\cdot,u_{n-1},\wh{q}_{n-1})\|_{1,2}^2 dt\\
  &\leq \wt{C} \mathbb{E} \bigg[\int_0^T e^{\lambda t}\|f(t,\cdot,0,0)\|_{1,2}^2 dt
  + \int_{0}^{T} e^{\lambda t}\Big( \|u_{n-1}\|_{1,2}^{2}
  + \|\wh{q}_{n-1}\|_{1,2}^{2} \Big) dt\bigg].
\end{split}\end{eqnarray}

If we denote the constant $C$ in \eqref{est:lambda} and \eqref{rmk:c01} by
$C_1$, then taking $\lambda_0 = 4\wt{C}+C_1+1$, we can prove a claim that for
each $n\geq 0$,
\begin{eqnarray}\label{proof:002}
  \mathbb{E}\int_{0}^{T} e^{\lambda_0 t}\Big( \|u_{n}\|_{1,2}^{2}
  + \|\wh{q}_{n}\|_{1,2}^{2} \Big) dt \leq 4 \mathbb{E} \bigg( e^{\lambda_0
  T}\|\vf\|_{1,2}^2 + \int_0^T e^{\lambda_0 t}\|f(t,\cdot,0,0)\|_{1,2}^{2} dt
  \bigg).
\end{eqnarray}
To prove it, the mathematical induction is used. Assume that
\eqref{proof:002} is true for $n-1$. Applying Lemma \ref{lem:lambda.est} to
equation \eqref{eq:iteration}, by \eqref{proof:003} we have
\begin{eqnarray*}\label{proof:005}
  \begin{split}
    &\mathbb{E}\int_{0}^{T} e^{\lambda_0 t}\Big( \|u_{n}\|_{1,2}^{2}
    + \|\wh{q}_{n}\|_{1,2}^{2} \Big) dt\\
    & \leq 2 e^{\lambda_0 T}\mathbb{E} \|\vf\|_{1,2}^{2} + \frac{2}{\lambda_0-C_1-1}
    \mathbb{E}\int_0^T e^{\lambda_0 t}\|f(t,\cdot,u_{n-1},\wh{q}_{n-1})\|_{1,2}^2 dt\\
    & \leq 2 e^{\lambda_0 T} \mathbb{E} \|\vf\|_{1,2}^{2}\\
    &~~~+ \frac{2\wt{C}}{\lambda_0-C_1-1}
    \mathbb{E} \bigg[\int_0^T e^{\lambda_0 t}\|f(t,\cdot,0,0)\|_{1,2}^2 dt + \int_{0}^{T} e^{\lambda_0 t}\Big( \|u_{n-1}\|_{1,2}^{2}
    + \|\wh{q}_{n-1}\|_{1,2}^{2} \Big) dt\bigg]\\
    & \leq 4 \mathbb{E} \bigg( e^{\lambda_0
    T}\|\vf\|_{1,2}^2 + \int_0^T e^{\lambda_0 t}\|f(t,\cdot,0,0)\|_{1,2}^{2} dt
    \bigg).
  \end{split}
\end{eqnarray*}

By \eqref{proof:002}, we immediately know that $\{(u_n,\wh{q}_n)\}_{n \geq
0}$ is uniformly bounded with the norm of $L^{2}_{\sP}W^{1,2} \times
L^{2}_{\sP}W^{1,2}(d')$. Hence there exist a subsequence $\{n'\}$ and a
function pair
$$(\widetilde{u},\widetilde{\wh{q}})
\in L^{2}_{\sP}W^{1,2} \times L^{2}_{\sP}W^{1,2}(d')$$ such that as $n'\rrow
\infty$,
\[ (u_{n'},\wh{q}_{n'})\rightharpoonup (\widetilde{u},\widetilde{\wh{q}})\ \
{\rm weakly}\ {\rm in}\ L^{2}_{\sP}W^{1,2} \times L^{2}_{\sP}W^{1,2}(d').\]

\emph{Step 3.} We then prove the strong convergence of $\{(u_n,\wh{q}_n)\}_{n
\geq 0}$ in $L^{2}_{\sP}W^{0,2} \times L^{2}_{\sP}W^{0,2}(d')$. In view of
\eqref{rmk:c01} and condition $(\F)$, taking $ \lambda=\lambda_1 = 8L^2 + C_1
+ 1$ and $n\geq1$ we have
\begin{eqnarray*}
  \begin{split}
    & \mathbb{E}\int_{0}^{T} e^{\lambda_1 t} \Big( \|u_{n+1}-u_{n}\|_{0,2}^{2}
    + \|\wh{q}_{n+1}-\wh{q}_{n}\|_{0,2}^{2} \Big) dt\\
    & \leq \frac{2}{\lambda_1-C_1-1}
    \mathbb{E}\int_{0}^{T} e^{\lambda_1 t}\|f(t,\cdot,u_{n},\wh{q}_{n})
    -f(t,\cdot,u_{n-1},\wh{q}_{n-1})\|_{0,2}^{2}dt\\
    & \leq \frac{1}{2}
    \mathbb{E}\int_{0}^{T} e^{\lambda_1 t} \Big( \|u_{n}-u_{n-1}\|_{0,2}^{2}
    + \|\wh{q}_{n}-\wh{q}_{n-1}\|_{0,2}^{2} \Big)dt,
  \end{split}
\end{eqnarray*}
which implies that $\{(u_{n},\wh{q}_{n})\}_{n\geq 0}$ is a Cauchy sequence in
the space $L^{2}_{\sP}W^{0,2} \times L^{2}_{\sP}W^{0,2}(d')$. Actually
$\{(u_{n},\wh{q}_{n}): n\geq 1\}$ is also a Cauchy sequence with the norm
$\mathbb{E}\int_{0}^{T}\|\cdot\|_{0,2}^{2}dt$ due to the norm equivalence
between $\sqrt{\mathbb{E}\int_{0}^{T} e^{\lambda_1 t}\|\cdot\|_{0,2}^{2}dt}$
and $\sqrt{\mathbb{E}\int_{0}^{T}\|\cdot\|_{0,2}^{2}dt}$ in
$L^{2}_{\sP}W^{0,2} \times L^{2}_{\sP}W^{0,2}(d')$. We denote the strong
limit of $\{(u_{n},\wh{q}_{n})\}_{n\geq 0}$ by $(u,\wh{q})$. Recalling the
subsequence $\{n'\}$ in step 2, we know that $\{(u_{n'},\wh{q}_{n'})\}$
converges strongly to $(u,\wh{q})$ in $L^{2}_{\sP}W^{0,2} \times
L^{2}_{\sP}W^{0,2}(d')$. By the uniqueness of the limit, we have
\[(u,\wh{q})= (\widetilde{u},\widetilde{\wh{q}})
\in L^{2}_{\sP}W^{1,2} \times L^{2}_{\sP}W^{1,2}(d').\]

\emph{Step 4.} Next we prove that $(u,\wh{q})$ is a solution of BSPDE
  \eqref{eq:main} to complete the existence proof. For this, we need verify that $(u,\wh{q})$ satisfies
\eqref{eq:sol2}. First we know that
\begin{eqnarray}\label{proof:004}
    \la u_{n'}(t),\eta\ra_{0}=&&\la\vf,\eta\ra_{0}+\int_{t}^{T}
    \la \widehat{\cL} u_{n'}(s) + \cM \wh{q}_{n'}(s)
    + f(s,x,u_{n'-1}(s),\wh{q}_{n'-1}(s)),
    \eta \ra_0 ds\nonumber\\
    &&-\int_{t}^{T}\la \wh{q}_{n'}(s)-u_{n',x}\sigma(s),
    \eta\ra_0 dW_{s}.
\end{eqnarray}
Since $(u_{n'-1},\wh{q}_{n'-1})$ converges strongly to $(u,\wh{q})$ in
$L^{2}_{\sP}W^{0,2} \times L^{2}_{\sP}W^{0,2}(d')$ as $n'\rrow \infty$, by
condition $(\F)$ it follows that, as $n'\rrow \infty$,
\[\mathbb{E} \int_0^T\|f(t,\cdot,u_{n'-1},\wh{q}_{n'-1})
-f(t,\cdot,u,\wh{q})\|_{0,2}^2(t)dt \longrightarrow 0.\]
Hence, for any $\eta \in C^{\infty}_{0}$, all terms of \eqref{proof:004}
converge weakly to the corresponding terms of \eqref{eq:sol2} in
$L^{2}_{\sP}(\Omega\times[0,T])$ since the operators of Lebesgue integration
and stochastic integration are continuous in
$L^{2}_{\sP}(\Omega\times[0,T])$. Therefore, $(u,\wh{q})$ is a generalized
solution of \eqref{eq:trans}. Setting $q = \wh{q}-u_x\sigma$, we know that
$(u,q)$ is a generalized solution of BSPDE \eqref{eq:main}.

Moreover, since $(u,\wh{q})$ is obtained, we regard $f(t,x,u,q+u_{x}\sigma)$
as the known coefficient and $(u,\wh{q})$ as the solution of linear BSPDE
with given $f(t,x,u,q+u_{x}\sigma)$. By condition $(\F)$,
$f(t,x,u,q+u_{x}\sigma)\in L^{2}_{\sP}W^{m,2}$. Then we get from Theorem
\ref{lem:l.eq.2} that $u\in L^2_{\sP}C_wW^{1,2}$ and \eqref{est:sl} follows.

\medskip

\emph{Step 5.} We finally deduce the uniqueness of solution to semi-linear
BSPDE. Assume that $(u_1,q_1)$ and $(u_2,q_2)$ are two generalized solutions
to BSPDE \eqref{eq:main}. Set
$\wh{q}_i = q_i + u_{i,x}\sigma$, $i=1,2$. Noticing \eqref{rmk:c01} and taking $\lambda= \lambda_1$ 
again, by condition $(\F)$ we have
\begin{eqnarray*}
    \mathbb{E}\int_{0}^{T} e^{\lambda_1 t} \Big( \|u_{1}-u_{2}\|_{0,2}^{2}
    + \|\wh{q}_{1}-\wh{q}_{2}\|_{0,2}^{2} \Big) dt
    \leq \frac{1}{2}
    \mathbb{E}\int_{0}^{T} e^{\lambda_1 t} \Big( \|u_{1}-u_{2}\|_{0,2}^{2}
    + \|\wh{q}_{1}-\wh{q}_{2}\|_{0,2}^{2} \Big)dt.
\end{eqnarray*}
The uniqueness of solution immediately follows, which completes the proof of
Theorem \ref{thm:sl.L2}.
\end{proof}
\medskip

In the remaining part of this section, the regularity of solution to
semi-linear BSPDE is explored. We consider a simpler form of BSPDE
\eqref{eq:main} with $f(t,x,v,r)$ independent of $r$:
\begin{numcases}{}\label{c4:sleq-1}
  \nonumber du = - [\cL u + \cM^k q^k + f(t,x,u)]dt + q^k dW^k_t\\
  u(T,x) = \vf(x),~~~~x\in\R^d.
\end{numcases}
For BSPDE \eqref{c4:sleq-1}, condition $(\F)$ is simplified as
follows:\medskip

$(\F')$~ the function $f(t,x,v)$ satisfies\\
(1) for arbitrary $(\omega,t,x,v)$, $f_x$ and $f_v$ exist;\\
(2) $f(\cdot,\cdot,0) \in L^{2}_{\sP}W^{1,2}$;\\
(3) there exists a constant $L>0$ such that for each $(\omega,t,x)$,
\begin{eqnarray*}
|f(t,x,v_1)-f(t,x,v_2)|+\|f_x(t,x,v_1)-f_x(t,x,v_2)\|\leq L |v_1 - v_2|,\
{\rm for}\ {\rm arbitrary}~v_1,v_2\in \R.
\end{eqnarray*}

Obviously, $f_v$ is bounded by the constant $L$.

We know from Theorem \ref{thm:sl.L2} that under conditions $(\A_1)$, $(\PC)$
  and $(\F')$, if $\vf\in
L^{2}_{\sF_T}(\Omega;W^{1,2})$, BSPDE \eqref{c4:sleq-1} has a unique solution
$(u,q)\in L^{2}_{\sP}C_wW^{1,2}\times L^{2}_{\sP}W^{0,2}(d')$. Moreover, some
regularity results for BSPDE \eqref{c4:sleq-1} can be obtained.
\begin{thm}\label{c4:bounded}
We assume that conditions \emph{$(\A_1)$}, \emph{$(\PC)$}
  and \emph{$(\F')$} are satisfied, and for $p\geq 2$, $f(\cdot,\cdot,0)\in L^{p }_{\sP}W^{1,p
}$ and $\vf\in L^{p }_{\sF_T}(\Omega;W^{1,p })$, then $u\in
L^{p}_{\sP}C_wW^{1,p }$ and there exists a constant $C(d,d',K_1,T,L,p)$ such
that
\begin{eqnarray}\label{c4:bounded.01}
  \mathbb{E}\sup _{t\le T}\|u(t)\|_{1,p }^p  \leq Ce^{Cp }
  \mathbb{E}\bigg(\|\vf\|_{1,p }^p  + \int_{0}^T
  \|f(t,\cdot,0)\|_{1,p }^p \,dt\bigg).
\end{eqnarray}
\end{thm}
\begin{proof}
By condition $(\F')$, it is easy to see that for arbitrary $v\in W^{1,p}$£¬
\begin{eqnarray}\label{c4:bounded.12}
  \|f(t,\cdot,v)\|^p_{1,p}=\|f(t,\cdot,v)\|^p_{0,p}
  +\|f_x(t,\cdot,v)\|^p_{0,p}\leq C(p)(\|f(t,\cdot,0)\|^p_{1,p} + L^p\|v\|^p_{1,p}).
\end{eqnarray}

To avoid heavy notation, we set
$$M_1 ~= ~ \mathbb{E} \bigg(\|\vf\|_{1,p }^p  + \int_{0}^T
  \|f(t,\cdot,0)\|_{1,p }^p \,dt\bigg).$$

Similar to arguments in Theorem \ref{thm:sl.L2}, we define a recursive
sequence $\{(u_{n},q_{n})\}_{n\geq 1}$ as follows:
\begin{numcases}{}\label{c4:eq:iteration-2}
  du_{n} = - \big[ \cL u_{n} + \cM q_{n}
  + f(t,x,u_{n-1})
  \big]\,dt + q_{n}\,dW_{t}\nonumber\\
  u_{n}(T) = \vf.\nonumber
\end{numcases}{}
If $u_{n-1}\in L^{p }_{\sP}W^{1,p }$, by \eqref{c4:bounded.12}
$f(\cdot,\cdot,u_{n-1})\in L^{p }_{\sP}W^{1,p }$, thus $u_{n}\in L^{p
}_{\sP}C_wW^{1,p }$ follows immediately from Theorem \ref{lem:l.eq.2}. By
setting $u_0=0$, we know from mathematical induction that
$\{u_{n}\}_{n\geq0}\subset L^{p}_{\sP}C_wW^{1,p}$. Furthermore, by the
estimate in Theorem \ref{lem:l.eq.2} and \eqref{c4:bounded.12}, we have for
arbitrary $t\in[0,T],~n\geq 1$,
\begin{eqnarray*}
  \mathbb{E} \|u_{n}(t)\|^p_{1,p}&&\leq C\,\mathbb{E}\bigg( \|\vf\|_{1,p }^p  + \int_{t}^T
  \|f(s,\cdot,u_{n-1})\|_{1,p }^p\,ds\bigg)\nonumber\\
  &&\leq C\,\int_{t}^T
  \mathbb{E}\|u_{n-1}(s)\|_{1,p }^p\,ds+C M_1,
\end{eqnarray*}
where $C$ is independent of $n$. A simple calculation leads to
\begin{eqnarray}\label{c4:bounded.13}
\mathbb{E}\|u_{n}(t)\|^p_{1,p} \leq C M_1
\sum_{k=0}^{n-1}\frac{1}{k\,!}C^k(T-t)^k \leq C M_1 e^{C(T-t)}.
\end{eqnarray}
Hence there exist a subsequence $\{n'\}$ and a function $u\in L^{2
}_{\sP}W^{1,2}$ such that as $n'\to\infty$, $u_{n'}$ converges weakly to $u$
in $L^{2 }_{\sP}W^{1,2}$. By Banach-Saks Theorem, we can construct a sequence
$u^k$ from finite convex combinations of $u_{n'}$ such that $u^k$ and $u^k_x$
converges to $u$ and $u_x$ for a.e. $t\in[0,T]$ $x\in\R^d$ a.s.,
respectively. Due to the norm itself is convex, \eqref{c4:bounded.13} implies
$$\mathbb{E} \int_0^T \|u^k(t)\|^p_{1,p}\,dt \leq C M_1(e^{CT}-1).$$
By Fatou Lemma, it turns out that
$$\mathbb{E} \int_0^T \|u(t)\|^p_{1,p}\,dt \leq C M_1(e^{CT}-1).$$
Regarding $u$ as the solution of linear BSPDE with given coefficient
$f(t,x,u)$, by \eqref{c4:bounded.12} and Theorem \ref{lem:l.eq.2} we obtain
\eqref{c4:bounded.01}.
\end{proof}

Form the proof of Theorem \ref{c4:bounded} and Corollary 2.3 in
\cite{du-ta-zh}, it is not hard to derive the following corollary.
\begin{cor}\label{dz12}
Let conditions \emph{$(\A_1)$}, \emph{$(\PC)$}
  and \emph{$(\F')$} be satisfied.
If $f(\cdot,\cdot,0)\in L^{\infty}_{\sP}W^{1,\infty},~\vf\in
L^{\infty}_{\sF_T}(\Omega;W^{1,\infty})$, then $u\in
L^{\infty}_{\sP}W^{1,\infty}$, i.e.
\begin{eqnarray*}
\|u\|_{L^{\infty}_{\sP}W^{1,\infty}}\leq
C(d,d_0,K_1,T,L,f(\cdot,\cdot,0),\vf)\triangleq C_{\infty}.
\end{eqnarray*}
\end{cor}
With the help of Sobolev's embedding theorem, it is not hard to deduce the
corollary below.
\begin{cor}\label{dz111}
Under the conditions in Theorem \ref{c4:bounded} with $p>2$ replaced by
$p>d$, $u(t,x)$ is jointly continuous on $(t,x)$ a.s.
\end{cor}

Based on Theorem \ref{c4:bounded}, we explore the regularity of solution to
BSPDE \eqref{c4:sleq-1}.
\begin{thm}\label{c4:bounded2} We assume that
\begin{description}
\item[(1)] conditions \emph{$(\A_2)$} and \emph{$(\PC)$}
  hold, and $\vf\in L^2_{\sF_T}(\Omega;W^{2,2})\cap
L^\infty_{\sF_T}(\Omega;W^{1,\infty})$;
\item[(2)] for arbitrary $(\omega,t,x,v)$, $f_x,f_v,f_{xx},f_{xv},f_{vv}$ exist;
\item[(3)] $f(\cdot,\cdot,0)\in L^2_{\sP}W^{2,2}\cap L^\infty_{\sP}W^{1,\infty}$;
\item[(4)] $f_v,f_{xv},f_{vv}$ are bounded by $L$;
\item[(5)] for arbitrary $(\omega,t,x,v)$, $|f_{xx}(t,x,v)|\leq |f_{xx}(t,x,0)| + L|v|$.
\end{description}
Then \eqref{c4:sleq-1} has a unique generalized
  solution $(u,q)$ satisfying
$$u \in L^2_{\sP}C_wW^{2,2}\cap L^\infty_{\sP}W^{1,\infty}\ {\rm and}\
q +u_x\sigma  \in L^2_{\sP}W^{2,2}.$$
\end{thm}
\begin{proof}
First of all, our assumptions satisfy the conditions in Theorem
\ref{c4:bounded} and Corollary \ref{dz12}, thus \eqref{c4:sleq-1} has a
unique solution $(u,q)$ satisfying
\begin{eqnarray*}
u \in L^2_{\sP}C_wW^{1,2}\cap L^\infty_{\sP}W^{1,\infty}\ {\rm and}\ q
+u_x\sigma  \in L^2_{\sP}W^{1,2}.
\end{eqnarray*}
To get a better regularity, for arbitrary $\delta>0$, we consider the
non-degenerate BSPDE below:
\begin{numcases}{}\label{c4:bounded.08}
    d u^{\delta} = -\big[(\delta \Delta + \cL) u^{\delta} + \cM q^{\delta}
    + f(t,x,u^{\delta})\big]\,dt + q^{\delta} \,dW_t\nonumber\\
    u^{\delta}(T) = \vf.\nonumber
\end{numcases}
By Theorem \ref{thm:sl.L2} we know that above BSPDE has a unique solution
$(u^{\delta},q^{\delta}) \in L^2_{\sP}W^{1,2}\times L^2_{\sP}W^{1,2}$, which
together with condition (4) leads to a fact that $f(t,x,u^{\delta})\in
L^2_{\sP}W^{1,2}$. Regarding $f(t,x,u^{\delta})$ as a given coefficient and
using Theorem 2.3 in \cite{du-me} for non-degenerate linear BSPDE, we can get
a better regularity of solution, i.e. $(u^{\delta},q^{\delta}) \in
L^2_{\sP}W^{3,2}\times L^2_{\sP}W^{2,2}$. Then $q^{\delta}
+u^{\delta}_x\sigma \in L^2_{\sP}W^{2,2}$, and by \eqref{c4:est:lambda-2}
there exists a positive constant $C_2(d,d',K_2,T)$ such that for any positive
number $\lambda >
  C_2+1$,
  \begin{eqnarray}\label{c4:bounded.05}
    &&\mathbb{E}\int_{0}^{T} e^{\lambda t}\Big( \|u^{\delta}\|_{2,2}^{2}
    + \|q^{\delta}+u^{\delta}_{x}\sigma\|_{2,2}^{2} \Big) \,dt\nonumber\\
    &&\leq 2 e^{\lambda T}\mathbb{E} \|\vf\|_{2,2}^{2}
    + \frac{2}{\lambda-C_2-1} \mathbb{E}\int_{0}^{T}
    e^{\lambda t}\|f(t,\cdot,u^{\delta})\|_{2,2}^{2}\,dt.
\end{eqnarray}
Also, by Corollary \ref{dz12} we have $|u_x|\leq C_\infty$, so it follows
from conditions (2)(4)(5) that
\begin{eqnarray*}
  &&|f(t,x,u^{\delta})|~\leq~ |f(t,x,0)| + L|u^{\delta}|,\\
  &&|\{f(t,x,u^{\delta})\}_x|~\leq~ |f_x(t,x,0)| + L(|u^{\delta}|+|u^{\delta}_x|),\\
  &&|\{f(t,x,u^{\delta})\}_{xx}|~\leq~|f_{xx}(t,x,u^{\delta})| + 2|f_{xv}(t,x,u^{\delta})|\cdot|u^{\delta}_x|+ |f_{vv}(t,x,u^{\delta})|\cdot|u^{\delta}_x|^2\\
  &&\ \ \ \ \ \ \ \ \ \ \ \ \ \ \ \ \ \ \ \ \ \ \ \ \ \ + |f_v(t,x,u^{\delta})|\cdot|u^{\delta}_{xx}|\\
  &&\ \ \ \ \ \ \ \ \ \ \ \ \ \ \ \ \ \ \ \ \ \ \leq~|f_{xx}(t,x,0)| + L|u^{\delta}| + (2 + C_\infty)L |u^{\delta}_x| + L|u^{\delta}_{xx}|.
\end{eqnarray*}
Hence,
$$\|f(t,\cdot,u^{\delta})\|_{2,2}^{2}\leq C(L,C_\infty)\Big(\|f(t,\cdot,0)\|_{2,2}^{2}
+ \|u^{\delta}\|_{2,2}^2\Big).$$ Putting this estimate into
\eqref{c4:bounded.05}, we immediately get
  \begin{eqnarray*}\begin{split}
    &\mathbb{E}\int_{0}^{T} e^{\lambda t}\Big( \|u^{\delta}\|_{2,2}^{2}
    + \|q^{\delta}+u^{\delta}_{x}\sigma\|_{2,2}^{2} \Big) \,dt\\
    &\leq 2 e^{\lambda T}\mathbb{E} \|\vf\|_{2,2}^{2}
    + \frac{2\,C(L,C_\infty)}{\lambda-C_2-1} \mathbb{E}\int_{0}^{T}
    e^{\lambda t}\Big(\|f(t,\cdot,0)\|_{2,2}^{2}
    + \|u^{\delta}\|_{2,2}^2 \Big)\,dt.
  \end{split}\end{eqnarray*}
Then taking $\lambda = 4\,C(L,C_\infty) + C_2 +1$ in above and setting
$$M_2~=~\mathbb{E} \bigg(\|\vf\|_{2,2}^{2}
    + \int_{0}^{T} \|f(t,\cdot,0)\|_{2,2}^{2} \,dt\bigg),$$
we obtain the uniformly bounded estimate for
$(u^{\delta},q^{\delta}+u^{\delta}_{x}\sigma)$ in $L^2_{\sP}W^{2,2}$, i.e.
\begin{eqnarray}\label{c4:bounded.09}\begin{split}
    \mathbb{E}\int_{0}^{T} \Big( \|u^{\delta}\|_{2,2}^{2}
    + \|q^{\delta}+u^{\delta}_{x}\sigma\|_{2,2}^{2} \Big) \,dt
    \leq 4 e^{\lambda T}M_2,
  \end{split}\end{eqnarray}
where $\lambda$ is independent of $\delta$. So we can get a sequence
$\{\delta_n\}\downarrow 0$ and $(\wh{u},\wh{r}) \in L^2_{\sP}W^{2,2}\times
L^2_{\sP}W^{2,2}$ such that
$(u^n,r^n)\triangleq(u^{\delta_n},q^{\delta_n}+u^{\delta_n}_{x}\sigma)$
converges weakly to $(\wh{u},\wh{r})$ in $L^2_{\sP}W^{2,2}\times
L^2_{\sP}W^{2,2}$. The weak convergence of $u^{\delta_n}$ to $\wh{u}$ in
$L^2_{\sP}W^{2,2}$ also implies the weak convergence of
$u^{\delta_n}_{x}\sigma$ to $\wh{u}_{x}\sigma$ in $L^2_{\sP}W^{1,2}$. Hence
$q^{\delta_n} = r^n - u^{\delta_n}_{x}\sigma$ converges weakly to
$\widehat{q}\triangleq\widehat{r}-\widehat{u}_x\sigma$ in $L^2_{\sP}W^{1,2}$.


Next we show that $\{u^{\delta_n}\}$ is a Cauchy sequence in
$L^2_{\sP}W^{0,2}$. If so, the strong convergence of $u^{\delta_n}$ to
$\wh{u}$ in $L^2_{\sP}W^{0,2}$ follows and it is easy to see that
$(\wh{u},\wh{q})$ is the unique solution to \eqref{c4:sleq-1} referring to
the arguments as in Theorem \ref{thm:sl.L2}.

To prove that $\{u^{\delta_n}\}$ is a Cauchy sequence in $L^2_{\sP}W^{0,2}$,
we set
$$u^{n,m} = u^{\delta_n}-u^{\delta_m},\ \ \ q^{n,m} = q^{\delta_n}-q^{\delta_m}.$$
Obviously, $(u^{n,m},q^{n,m})$ satisfies equations as follows:
\begin{numcases}{}\label{c4:bounded.10}
  d u^{n,m} = -\big\{(\,\delta_n \,\Delta + \cL)\, u^{n,m} + \cM q^{n,m}
  + f(t,x,u^{\delta_n})-f(t,x,u^{\delta_m}) \nonumber\\
  ~~~~~~~~~~~~~~~~+ (\delta_n - \delta_m)\,\Delta u^{\delta_m} \big\}\,dt
  + q^{n,m} \,dW_t\nonumber\\
  u^{n,m}(T) = 0.\nonumber
\end{numcases}
By \eqref{est:l.eq.2} in the case of $m_1=0$ and \eqref{c4:bounded.09}, for
$t\in[0,T]$, we have
\begin{eqnarray*}
  &&\mathbb{E} \|u^{n,m}(t)\|_{0,2}^2 \\
  &&\leq~ C\,
  \mathbb{E}\bigg\{\int_t^T \|f(s,\cdot,u^{n})-f(s,\cdot,u^{m})\|_{0,2}^2\,ds
  +(\delta_n - \delta_m)\int_s^T \|\Delta u^{m}(s)\|_{0,2}^2\,ds\bigg\}\\
  &&\leq~ C\,
  \mathbb{E}\int_t^T \|u^{n,m}(s)\|_{0,2}^2\,ds
  +(\delta_n - \delta_m)\,CM_2,
\end{eqnarray*}
where the constant $C$ is independent of $\delta_n, \delta_m$. Therefore, we
can apply Gronwall inequality and take $n,m\rrow\infty$ to deduce that
$\{u^n\}$ is a Cauchy sequence in $L^2_{\sP}W^{0,2}$. The proof of Theorem
\ref{c4:bounded2} is complete.
\end{proof}

\begin{rmk}
(i) Theorems \ref{c4:bounded} and \ref{c4:bounded2} improve much in many aspects in comparison with Theorems 3.2 and 5.1 in Hu-Ma-Yong \cite{hu-ma-yo}. For example, our result includes multi-dimensional equation and the coefficients $\sigma,\nu$ in BSPDE can depend on $x$ (actually, all the coefficients in our setting are a function of $(\omega,t,x)$). Also the regularity condition of coefficients in Theorem \ref{c4:bounded2} is weaker than that in Theorem 5.1 in \cite{hu-ma-yo}. Needless to say, all these improvements are not trivial.\\
(ii) Denote by $D_x^iD_v^jf,\ i,j\in\mathbb{Z}^+$ the derivative of $f$ which
is $i$ order with respect to $x$ and $j$ order with respect to $v$. For
$m\geq 1$, if we assume
\begin{description}
\item[(1)]conditions \emph{$(\A_m)$} and \emph{$(\PC)$}
 hold, and $\vf\in L^2_{\sF_T}(\Omega;W^{m,2})\cap
L^\infty_{\sF_T}(\Omega;W^{m-1,\infty})$;
\item[(2)]for arbitrary $(\omega,t,x,v)$, all $D_x^iD_v^jf$ exist, where $0\leq i,j\leq m$ and $i+j>0$;
\item[(3)]$f(\cdot,\cdot,0)\in L^2_{\sP}W^{m,2}\cap L^\infty_{\sP}W^{m-1,\infty}$;
\item[(4)]all $D_x^iD_v^jf$ are bounded by $L$, where $0\leq i\leq m-1, 0\leq j\leq m$ and $i+j>0$;
\item[(5)]for arbitrary $(\omega,t,x,v)$, $|D^m_xf(t,x,v)|\leq |D^m_xf(t,x,0)| + L|v|$.
\end{description}
Then from the argument of Theorem \ref{c4:bounded2}, it is not hard to prove
that \eqref{c4:sleq-1} has a unique generalized
  solution $(u,q)$ satisfying
$$u \in L^2_{\sP}C_wW^{m,2}\cap L^\infty_{\sP}W^{m-1,\infty}\ {\rm and}\
q +u_x\sigma  \in L^2_{\sP}W^{m,2}.$$

\end{rmk}

\section{Connection between BSPDEs and FBSDEs}
\setcounter{equation}{0}

\ \ \ \ In this section, we study the connection between semi-linear BSPDEs
and FBSDEs. This kind of connection is established in a non-Markov frame and
can be regarded as an extension of Feynman-Kac formula for semi-linear PDEs
and BSDEs (c.f. \cite{pa-pe, pe}).

First give a BSDE whose coefficients may be non-Markovian:
\begin{numcases}{}\label{eq:sde}
  X^{t,x}_s = x + \int_{t}^{s} b(r,X^{t,x}_r)dr + \int_{t}^{s} \sigma(r,X^{t,x}_r)dW_r,\ \ \ s\geq t,\nonumber\\
  X^{t,x}_s=x,\ \ \ 0\leq s<t.
\end{numcases}
where $W_s = (W^{1}_s,\cdots,W^{d'}_s)^{*}$. We always assume that $b,\sigma$
satisfy $(\A_1)$. The BSDE coupled with above forward SDE is usually called
FBSDE:
\begin{eqnarray}\label{eq:bsde}
  Y^{t,x}_s = \vf(X^{t,x}_T) + \int_{s}^{T} f(r,X^{t,x}_r,Y^{t,x}_r) dr - \int_{s}^{T} Z^{t,x}_r dW_r.
\end{eqnarray}
\begin{rmk}\label{29} (i) Given $p >d$, note that
\begin{eqnarray*}
\mathbb{E}\int_0^T|f(s,X^{t,x}_s,0)|^pds\leq
\mathbb{E}\int_0^T\|f(s,\cdot,0)\|_{L^\infty}^pds\leq
C\mathbb{E}\int_0^T\|f(s,\cdot,0)\|_{W^{1,p}}^pds.
\end{eqnarray*}
Hence, if $f(\cdot,\cdot,0)\in L^{p}_{\sP}W^{1,p}$ , $\vf\in
L^{p}_{\sF_T}(\Omega;W^{1,p})$, and for any $\omega\in\Omega$, $s\in[0,T]$,
$f(s,x,y)$ satisfies the uniformly Lipschitz condition with respect to $y$,
we can use It$\hat { o}$ formula and the localization procedure to prove that
there exists a unique $(Y^{t,x}_s,Z^{t,x}_s)_{s\in[t,T]}$ which satisfies the
form \eqref{eq:bsde} and
\begin{eqnarray}\label{dz42}
\mathbb{E}\sup_{s\in[t,T]}\int_t^T|Y^{t,x}_s|^pds+\mathbb{E}\int_t^T|Y^{t,x}_s|^{p-2}|Z^{t,x}_s|^2ds+\Big(\mathbb{E}\int_t^T|Z^{t,x}_s|^2ds\Big)^{p\over2}<\infty.
\end{eqnarray}
One can refer to e.g. Lemma 4.3 in \cite{qzh} for the localization procedure, and in order to save the space we leave out the localization procedure arguments in this section.\\
(ii) For $s\in[0,t]$, \eqref{eq:bsde} is equivalent to the following FBSDE:
\begin{eqnarray*}\label{dz41}
Y_s^{x}=Y_t^{t,x}+\int_{s}^{t}f(r,x,Y^{x}_r)dr-\int_{s}^{t} Z^{x}_rdW_r.
\end{eqnarray*}
As stated in (i), in view of $Y_t^{t,x}\in L^{p}_{\sF_T}(\Omega;L^p)$, the
above equation has a unique solution $(Y^{x}_s,Z^{x}_s)_{s\in[0,t]}$. To
unify the notation, we define
$({Y}_s^{t,x},{Z}_s^{t,x})=({Y}_s^{x},{Z}_s^{x})$ when $s\in[0,t)$.
\end{rmk}

Our purpose is to investigate the connection between FBSDE \eqref{eq:bsde}
and the following BSPDE:
\begin{numcases}{}\label{eq:bspde}
  du = -\big[\alpha^{ij}u_{x^ix^j} + b^{i}u_{x^i}
  + \sigma^{ik}q^{k}_{x^i} + f(t,x,u
  ) \big]dt + q^{k}dW^{k}_t\nonumber\\
  u(T,x) = \vf(x),~~~~x\in \R^{d},
\end{numcases}
where $\alpha^{ij} = \frac{1}{2}\sigma^{ik}\sigma^{jk}$.

We begin with the linear case that $f(t,x,y,z) = c(t,x)y + \nu^{k}(t,x)z^{k}
+ F(t,x)$ and in this case FBSDE has a form like below:
\begin{eqnarray}\label{dz20}
  Y^{t,x}_s = \vf(X^{t,x}_T) + \int_{s}^{T}\big[ c(r,X^{t,x}_r)Y^{t,x}_r + \nu(r,X^{t,x}_r)Z^{t,x}_r + F(r,X^{t,x}_r) \big] dr - \int_{s}^{T} Z^{t,x}_r dW_r.\ \ \ \
\end{eqnarray}
The corresponding linear BSPDE is as follows:
\begin{numcases}{}\label{dz17}
  du = - \big[ \cL u + \cM q + F \big]dt
  + q^{k} dW^{k}_{t}\nonumber\\
u(T,x) = \vf(x),\quad x\in \R^{d}.
\end{numcases}
Referring to Lemma 4.5.6 in \cite{ku}, we first give a useful lemma.
\begin{lem}\label{25} Under condition \emph{$(\A_1)$}, for $p\geq1$, $t',t\in[0,T]$,
the stochastic flow defined by (\ref{eq:sde}) satisfies
\begin{eqnarray*}
\mathbb{E}\sup_{s\in[0,T]}|X_s^{t',x'}-X_s^{t,x}|^{2p}\leq
C(p,T)\Big(1+|x|^{2p}+|x'|^{2p}\Big)\Big(|x'-x|^{2p}+|t'-t|^{p}\Big)\ \ \
\rm{a.s.}
\end{eqnarray*}
\end{lem}

The following proposition borrows ideas from \cite{pa-pe, qzh}. Although
$F(s,x)$ is not Lipschitz continuous on $x$, we still can derive the
continuity of $Y^{t,x}_t$ since the H$\ddot{\rm o}$lder continuity of
$F(s,x)$ on $x$.
\begin{prop}\label{26} Let conditions \emph{$(\A_1)$} be satisfied. For a given $p>2d+2$, suppose
  $F\in L^{p}_{\sP}W^{1,p}$ and $\vf\in L^{p}_{\sF_T}(\Omega;W^{1,p})$. If $(Y_{s}^{t,x})_{s\in[t,T]}$ is the solution of FBSDE (\ref{dz20}), then for $t\in[0,T]$, $x\in\mathbb{R}^d$, $(t,x)\longrightarrow Y_{t}^{t,x}$ is a.s. continuous.
\end{prop}
\begin{proof} By Remark \ref{29}, we know that FBSDE (\ref{dz20}) has a unique
solution $(Y^{t,x}_s,Z^{t,x}_s)_{s\in[0,T]}$ and it satisfies \eqref{dz42}.
For $t,t'\in[0,T]$, $x,x'\in\mathbb{R}^d$, $s\geq0$, $0<\beta<1$, assuming
without loss of any generality that $\beta p>2d+2$ and $|x-x'|\le 1$, we have
\begin{eqnarray}\label{dz34}
  &&\mathbb{E}\int_t^T{\rm e}^{\beta pKr}|Y_{s}^{t',x'}-Y_{s}^{t,x}|^{\beta p-2}|c(s,X^{t',x'}_s)Y^{t',x'}_s-c(s,X^{t,x}_s)Y^{t,x}_s|^{2}ds\nonumber\\
  &&\leq2\mathbb{E}\int_t^T{\rm e}^{\beta pKr}|Y_{s}^{t',x'}-Y_{s}^{t,x}|^{\beta p-2}\Big(|c(s,X^{t',x'}_s)|^2|Y^{t',x'}_s-Y^{t,x}_s|^2\nonumber\\
  &&\ \ \ \ \ \ \ \ \ \ \ \ \ \ \ \ \ \ \ \ \ \ \ \ \ \ \ \ \ \ \ \ \ \ \ \ \ \ \ \ \ \ \ \ \ \ +|c(s,X^{t',x'}_s)-c(s,X^{t,x}_s)|^2|Y^{t,x}_s|^2\Big)ds\nonumber\\
   &&\leq2K_1^2\mathbb{E}\int_t^T{\rm e}^{\beta pKr}|Y^{t',x'}_s-Y^{t,x}_s|^{\beta p}ds\nonumber\\
   &&\ \ \ +2K_1^2\mathbb{E}\int_t^T{\rm e}^{\beta pKr}|Y^{t',x'}_s-Y^{t,x}_s|^{\beta p-2}|X^{t,x,\eps}_s-X^{t,x}_s|^2|Y^{t,x}_s|^2ds\nonumber\\
 &&\leq2K_1^2\mathbb{E}\int_t^T{\rm e}^{\beta pKr}|Y^{t',x'}_s-Y^{t,x}_s|^{\beta p}ds+\eps\mathbb{E}\int_t^T{\rm e}^{\beta pKr}|Y^{t',x'}_s-Y^{t,x}_s|^{\beta p}ds\nonumber\\
 &&\ \ \ +C\Big(\mathbb{E}\int_t^T|Y^{t,x}_s|^{p}ds\Big)^{\beta}\Big(\mathbb{E}\int_t^T|X^{t',x'}_s-X^{t,x}_s|^{{\beta p}\over{1-\beta}}ds\Big)^{1-\beta},
\end{eqnarray}
where $\eps>0$ is a generic constant which can be taken sufficiently small.
Similarly, it follows that
\begin{eqnarray}\label{dz35}
&&\mathbb{E}\int_t^T{\rm e}^{\beta pKr}|Y_{s}^{t',x'}-Y_{s}^{t,x}|^{\beta p-2}|\nu(s,X^{t',x'}_s)Z^{t',x'}_s-\nu(s,X^{t,x}_s)Z^{t,x}_s|^2ds\nonumber\\
&&\leq2K_1^2\mathbb{E}\int_t^T{\rm e}^{\beta pKr}|Y^{t',x'}_s-Y^{t,x}_s|^{\beta p-2}|Z^{t',x'}_s-Z^{t,x}_s|^{2}ds\nonumber\\
&&\ \ \ +2K_1^2\mathbb{E}\int_t^T{\rm e}^{\beta
pKr}|Y^{t',x'}_s-Y^{t,x}_s|^{\beta
p-2}|X^{t',x'}_s-X^{t,x}_s|^2|Z^{t,x}_s|^2ds.
\end{eqnarray}
Noticing the $C^{0,\alpha}$ norm is controlled by the $W^{1,p}$ norm in view
of Sobolev embedding theorem, where $\alpha=1-{d\over p}<1$, we have
\begin{eqnarray}\label{dz36}
&&\mathbb{E}\int_t^T{\rm e}^{\beta pKr}|F(s,X^{t',x'}_s)-F(s,X^{t,x}_s)|^{\beta p}ds\nonumber\\
&&\leq\mathbb{E}\int_t^T{\rm e}^{\beta pKr}\|F(s,\omega)\|^{\beta p}_{C^{0,\alpha}}|X^{t',x'}_s-X^{t,x}_s|^{\alpha\beta p} ds\nonumber\\
&&\leq\Big(\mathbb{E}\int_t^T{\rm e}^{\beta pKr}\|F(s,\omega)\|^p_{W^{1,p}}ds\Big)^{\beta }\Big(\mathbb{E}\int_t^T{\rm e}^{\beta pKr}|X^{t',x'}_s-X^{t,x}_s|^{{\alpha\beta p}\over{1-\beta}}ds\Big)^{1-\beta}\nonumber\\
&&\leq C\Big(\mathbb{E}\int_t^T|X^{t',x'}_s-X^{t,x}_s|^{{\alpha\beta
p}\over{1-\beta}}ds\Big)^{1-\beta}.
\end{eqnarray}
Similar to above, we can also get
\begin{eqnarray}\label{dz37}
\mathbb{E}{\rm e}^{\beta pKr}|\varphi(X^{t',x'}_T)-\varphi(X^{t,x}_T)|^{\beta
p}\leq C\Big(\mathbb{E}|X^{t',x'}_T-X^{t,x}_T|^{{\alpha\beta
p}\over{1-\beta}}\Big)^{1-\beta}.
\end{eqnarray}
Now applying It$\hat {\rm o}$'s formula to ${\rm e}^{\beta
pKs}|Y_{s}^{t',x'}-Y_{s}^{t,x}|^{\beta p}$, we have
\begin{eqnarray}\label{dz38}
&&{\rm e}^{\beta pKs}|Y_{s}^{t',x'}-Y_{s}^{t,x}|^{\beta p}+{\beta pK}\int_{s}^{T}{\rm e}^{\beta pKr}|Y_{r}^{t',x'}-Y_{r}^{t,x}|^{\beta p}dr\nonumber\\
&&+{{{\beta p}({\beta p}-1)}\over2}\int_{s}^{T}{\rm e}^{\beta pKr}|Y_{r}^{t',x'}-Y_{r}^{t,x}|^{\beta p-2}|Z_{r}^{t',x'}-Z_{r}^{t,x}|^2dr\nonumber\\
&&\leq{\rm e}^{\beta
pKT}|\varphi(X_{T}^{t',x'})-\varphi(X_{T}^{t',x'})|^{\beta p}+({\beta
p}+2K_1^2{\beta p})\int_{s}^{T}{\rm e}^{\beta pKr}
|Y_{r}^{t',x'}-Y_{r}^{t,x}|^{\beta p}dr\nonumber\\
&&\ \ \ +{{\beta p}\over2}\int_{s}^{T}{\rm e}^{\beta pKr}|Y_{r}^{t',x'}-Y_{r}^{t,x}|^{\beta p-2}|c(r,X^{t',x'}_r)Y^{t',x'}_r-c(r,X^{t,x}_r)Y^{t,x}_r|^2dr\nonumber\\
&&\ \ \ +{{\beta p}\over{8K_1^2}}\int_{s}^{T}{\rm e}^{\beta pKr}|Y_{r}^{t',x'}-Y_{r}^{t,x}|^{\beta p-2}|\nu(r,X^{t',x'}_r)Z^{t',x'}_r-\nu(r,X^{t,x}_r)Z^{t,x}_r|^2dr\nonumber\\
&&\ \ \ +{\beta p}\int_{s}^{T}{\rm e}^{\beta pKr}
|Y_{r}^{t',x'}-Y_{r}^{t,x}|^{\beta p-2}|F(r,X^{t',x'}_r)-F(r,X^{t,x}_r)|^2dr\nonumber\\
&&\ \ \ -{{\beta p}\over2}\int_{s}^{T}{\rm e}^{\beta
pKr}(Y_{r}^{t',x'}-Y_{r}^{t,x})^{\beta
p-2}(Y_{r}^{t',x'}-Y_{r}^{t,x})(Z_{r}^{t',x'}-Z_{r}^{t,x})dW_r.
\end{eqnarray}
Taking expectation on both sides of \eqref{dz38}, by
(\ref{dz34})-(\ref{dz37}) we have
\begin{eqnarray}\label{dz39}
&&({\beta pK-\beta p-3K_1^2\beta p}-\eps)\mathbb{E}\int_{s}^{T}{\rm e}^{\beta pKr}|Y_{r}^{t',x'}-Y_{r}^{t,x}|^{\beta p}dr\nonumber\\
&&+{{{\beta p}(2{\beta p}-3)}\over4}\mathbb{E}\int_{s}^{T}{\rm e}^{\beta pKr}|Y_{r}^{t',x'}-Y_{r}^{t,x}|^{\beta p-2}|Z_{r}^{t',x'}-Z_{r}^{t,x}|^2dr\nonumber\\
&&\leq C\Big(\mathbb{E}|X^{t',x'}_T-X^{t,x}_T|^{{\alpha\beta p}\over{1-\beta}}\Big)^{1-\beta}+C\mathbb{E}\Big(\int_s^T|X^{t',x'}_r-X^{t,x}_r|^{{\beta p}\over{1-\beta}}dr\Big)^{1-\beta}\nonumber\\
&&\ \ \ +{{\beta p}\over4}\mathbb{E}\int_s^T{\rm e}^{\beta pKr}|Y^{t',x'}_r-Y^{t,x}_r|^{\beta p-2}|X^{t',x'}_r-X^{t,x}_r|^2|Z^{t,x}_r|^2dr\nonumber\\
&&\ \ \ +C\Big(\mathbb{E}\int_s^T|X^{t',x'}_r-X^{t,x}_r|^{{\alpha\beta
p}\over{1-\beta}}dr\Big)^{1-\beta}.
\end{eqnarray}
Then applying B-D-G inequality to \eqref{dz38} and using \eqref{dz39} with a
sufficiently large $K$, we have
\begin{eqnarray*}
&&\mathbb{E}\sup_{s\in[0,T]}|Y_{s}^{t',x'}-Y_{s}^{t,x}|^{\beta p}\nonumber\\
&&\leq C\mathbb{E}\int_{0}^{T}|Y_{s}^{t',x'}-Y_{s}^{t,x}|^{\beta p}ds+C\mathbb{E}\int_{0}^{T}|Y_{s}^{t',x'}-Y_{s}^{t,x}|^{\beta p-2}|Z_{s}^{t',x'}-Z_{s}^{t,x}|^2ds\nonumber\\
&&\ \ \ +C\Big(\mathbb{E}|X^{t',x'}_T-X^{t,x}_T|^{{\alpha\beta p}\over{1-\beta}}\Big)^{1-\beta}+C\mathbb{E}\Big(\int_0^T|X^{t',x'}_s-X^{t,x}_s|^{{\beta p}\over{1-\beta}}ds\Big)^{1-\beta}\nonumber\\
&&\ \ \ +C\mathbb{E}\int_0^T|Y^{t',x'}_s-Y^{t,x}_s|^{\beta p-2}|X^{t',x'}_s-X^{t,x}_s|^2|Z^{t,x}_s|^2ds+C\Big(\mathbb{E}\int_0^T|X^{t',x'}_s-X^{t,x}_s|^{{\alpha\beta p}\over{1-\beta}}ds\Big)^{1-\beta}\nonumber\\
&&\leq 
C\Big(\mathbb{E}|X^{t',x'}_T-X^{t,x}_T|^{{\alpha\beta p}\over{1-\beta}}\Big)^{1-\beta}+C\mathbb{E}\Big(\int_0^T|X^{t',x'}_s-X^{t,x}_s|^{{\beta p}\over{1-\beta}}ds\Big)^{1-\beta}\nonumber\\
&&\ \ \ +\eps\mathbb{E}\sup_{s\in[0,T]}|Y^{t',x'}_s-Y^{t,x}_s|^{\beta p}+C\mathbb{E}\sup_{s\in[0,T]}|X^{t,x,\eps}_s-X^{t,x}_s|^{\beta p}\Big(\int_0^T|Z^{t,x}_s|^2ds\Big)^{\beta p\over2}\nonumber\\
&&\ \ \ +C\Big(\mathbb{E}\int_0^T|X^{t',x'}_s-X^{t,x}_s|^{{\alpha\beta
p}\over{1-\beta}}ds\Big)^{1-\beta}.
\end{eqnarray*}
Hence,
\begin{eqnarray*}
&&(1-\eps)\mathbb{E}\sup_{s\in[0,T]}|Y_{s}^{t',x'}-Y_{s}^{t,x}|^{\beta p}\nonumber\\
&&\leq C\Big(\mathbb{E}|X^{t',x'}_T-X^{t,x}_T|^{{\alpha\beta p}\over{1-\beta}}\Big)^{1-\beta}+C\mathbb{E}\Big(\int_0^T|X^{t',x'}_s-X^{t,x}_s|^{{\beta p}\over{1-\beta}}ds\Big)^{1-\beta}\nonumber\\
&&\ \ \ +C\Big(\mathbb{E}\sup_{s\in[0,T]}|X^{t,x,\eps}_s-X^{t,x}_s|^{{\beta p}\over{1-\beta}}\Big)^{1-\beta}\Big(\mathbb{E}\Big(\int_0^T|Z^{t,x}_s|^2ds\Big)^{p\over2}\Big)^\beta\nonumber\\
&&\ \ \ +C\Big(\mathbb{E}\int_0^T|X^{t',x'}_s-X^{t,x}_s|^{{\alpha\beta
p}\over{1-\beta}}ds\Big)^{1-\beta}.
\end{eqnarray*}
By Lemma \ref{25}, it yields that
\begin{eqnarray*}\label{dz40}
\mathbb{E}\sup_{s\in[0,T]}|Y_{s}^{t',x'}-Y_{s}^{t,x}|^{\beta p}\leq
C(p,T)\Big(1+|x|^{p}+|x'|^{p}\Big)\Big(|x'-x|^{\alpha\beta
p}+|t'-t|^{{\alpha\beta p}\over2}\Big)\ \ \rm{a.s.}\ \ \ \
\end{eqnarray*}

Since $\beta p>2d+2$, by Kolmogorov continuity theorem (see e.g. Theorem
1.4.1 in \cite{ku}) we know that $Y_{s}^{(\cdot,\cdot)}$ has a continuous
modification for $t\in[0,T]$ and $x\in\bar{B}(0,R)$ with the norm
$\sup_{s\in[0,T]}|Y_{s}^{(\cdot,\cdot)}|$, where $\bar{B}(0,R)$ is the closed
ball in $\mathbb{R}^d$ with the center $0$ and the radius $R\in\mathbb{Z}^+$.
In particular,
\begin{eqnarray*}
\lim_{t'\rightarrow t\atop x'\rightarrow x}|Y_{t'}^{t',x'}-Y_{t'}^{t,x}|=0.
\end{eqnarray*}
Thus we have
\begin{eqnarray*}
\lim_{t'\rightarrow t\atop x'\rightarrow
x}|Y_{t'}^{t',x'}-Y_{t}^{t,x}|\leq\lim_{t'\rightarrow t\atop x'\rightarrow
x}(|Y_{t'}^{t',x'}-Y_{t'}^{t,x}|+|Y_{t'}^{t,x}-Y_{t}^{t,x}|)=0\ \ \ {\rm
a.s.}
\end{eqnarray*}
The convergence of the second term follows from the continuity of
$Y_{s}^{t,x}$ in $s$. That is to say $Y_{t}^{t,x}$ is a.s. continuous,
therefore $Y_{t}^{t,x}$ is continuous with respect to $t\in[0,T]$ and
$x\in\bar{B}(0,R)$ on a full-measure set $\Omega^R$. Taking
$\tilde{\Omega}=\bigcap_{R\in\mathbb{Z}^+}\Omega^R$, we have
$P(\tilde{\Omega})=1$. Since
$\bigcup_{R\in\mathbb{Z}^+}\bar{B}(0,R)=\mathbb{R}^d$, for any $t\in[0,T]$
and $x\in\mathbb{R}^d$, there exists an $R$ s.t. $x\in\bar{B}(0,R)$. On the
other hand, For any $\omega\in\tilde{\Omega}$, obviously
$\omega\in\Omega^{R}$, $R=1,2,\cdots$. So $Y_{t}^{t,x}$ is continuous with
respect to $t\in[0,T]$ and $x\in\mathbb{R}^{d}$ on $\tilde{\Omega}$.
Proposition \ref{26} is proved.
\end{proof}

Then we can get the correspondence between BSPDE and FBSDE in the linear
case.
\begin{thm}\label{27}
  Let conditions \emph{$(\A_1)$} and \emph{$(\PC)$} be satisfied. For a given $p>2d+2$, suppose
  $F\in L^{p}_{\sP}W^{1,p}$ and $\vf\in L^{p}_{\sF_T}(\Omega;W^{1,p})$, then
  the solution $(u,q)$ to BSPDE \eqref{dz17} satisfies
  \begin{eqnarray}\label{dz19}
  u(s,X^{t,x}_s)=Y^{t,x}_s
  \ \ {\rm for}\ {\rm all}\ s\in[t,T],\ x\in\R^d\ {\rm a.s.},
  \end{eqnarray}
where $X$ and $(Y,Z)$ are the solutions of SDE \eqref{eq:sde} and FBSDE
\eqref{dz20}, respectively.
\end{thm}
\begin{proof}
\emph{Step 1. }First we smootherize all the coefficients in \eqref{dz20} and
\eqref{dz17}. For this, take a nonnegative function $\rho \in
  C_{0}^{\infty}(\R^d,\R^1)$
  such that $\int_{\R^d}\rho(x)dx =1$.
  For arbitrary $\eps > 0$ and a mapping $h:\R^d\longrightarrow\R^1$, we define $h^{\eps}$ by
  \begin{eqnarray*}\label{eq:S_eps}
    h^{\eps}(x) = \eps^{-d} \rho\left(\frac{x}{\eps}\right) \ast h(x)\ \ \ {\rm for}\ x\in\mathbb{R}^d.
  \end{eqnarray*}
Moreover, if $h$ is a vector or matrix, we get the smootherized $h^\eps$ by
smootherizing each element in $h$. In this way, we can smootherize all the
coefficients and get two equations with smootherized coefficients:
\begin{eqnarray}\label{dz21}
  Y^{t,x,\eps}_s& = &\vf^{\eps}(X^{t,x,\eps}_T) + \int_{s}^{T}\big[ c^{\eps}(r,X^{t,x,\eps}_r)Y^{t,x,\eps}_r + \nu^{\eps}(r,X^{t,x,\eps}_r)Z^{t,x,\eps}_r + F^{\eps}(r,X^{t,x,\eps}_r) \big] dr \nonumber\\
  &&- \int_{s}^{T} Z^{t,x,\eps}_r dW_r
\end{eqnarray}
and
\begin{numcases}{}\label{dz22}
  du^\eps(t,x) = - \big[ \cL^\eps u^\eps(t,x) + \cM^\eps q ^\eps(t,x)+ F^\eps(t,x) \big]dt
  + q^{\eps}(t,x) dW_{t}\nonumber\\
u^{\eps}(T,x) = \vf^{\eps}(x),\quad x\in \R^{d},
\end{numcases}
where $(Y^{\eps},Z^{\eps})$ and $(u^\eps,q^\eps)$ are the unique solutions of
\eqref{dz21} and \eqref{dz22}, respectively. Due to the smooth coefficients,
we know that all $X_s^{t,x,\eps}$, $Y_s^{t,x,\eps}$ and $u^\eps(t,x)$ have a
high regularity on variable $x$ such that $u^\eps(s,X^{t,x,\eps}_s)$,
$Y_s^{t,x,\eps}\in C^{1,2}([0,T]\times\R^d)$. By It$\hat {\rm o}$-Wentzell
formula it is not hard to deduce that $u^\eps(s,X^{t,x,\eps}_s)$ is also a
solution to FBSDE \eqref{dz21}. Due to the uniqueness of solution, we have
$u^\eps(s,X^{t,x,\eps}_s)=Y_s^{t,x,\eps}$ for a.e. $t\in[0,T]$, $x\in\R^d$
a.s., then the continuity with respect to $t,x$ ensures that this equality is
true for all $t\in[0,T]$, $x\in\R^d$ in a full measure set in $\Omega$.

\emph{Step 2. }We then prove that as for a.e. $x\in\mathbb{R}^d$,
\begin{eqnarray}\label{dz29}
 \mathbb{E}\sup_{s\in[t,T]}|Y^{t,x,\eps}_s-Y^{t,x}_s|^2\longrightarrow0,\ \ \ {\rm as}\ \eps\to0.
\end{eqnarray}
First noting condition $(\A_1)$ and the construction of convolution we have
\begin{eqnarray*}
  &&\lim_{\eps\to0}\mathbb{E}\int_t^T|b^\eps(s,X^{t,x,\eps}_s)-b(s,X^{t,x}_s)|^2ds\nonumber\\
  &&\leq\lim_{\eps\to0}2\mathbb{E}\int_t^T\Big(|b^\eps(s,X^{t,x,\eps}_s)-b(s,X^{t,x,\eps}_s)|^2+|b(s,X^{t,x,\eps}_s)-b(s,X^{t,x}_s)|^2\Big)ds\nonumber\\
  &&\leq\lim_{\eps\to0}2\mathbb{E}\int_t^T\Big(\sup_{y\in\mathbb{R}^d}|b^\eps(s,y)-b(s,y)|^2+K_1^2|X^{t,x,\eps}_s-X^{t,x}_s|^2\Big)ds\nonumber\\
  &&=\lim_{\eps\to0}2K_1^2\mathbb{E}\int_t^T|X^{t,x,\eps}_s-X^{t,x}_s|^2ds.
\end{eqnarray*}
A similar calculation leads to
\begin{eqnarray*}
\lim_{\eps\to0}\mathbb{E}\int_t^T|\sigma^\eps(s,X^{t,x,\eps}_s)
-\sigma(s,X^{t,x}_s)|^2ds\leq\lim_{\eps\to0}2K_1^2\mathbb{E}\int_t^T|X^{t,x,\eps}_s-X^{t,x}_s|^2ds.
\end{eqnarray*}
Hence applying It\^o's formula and B-D-G inequality we have
\begin{eqnarray*}
\lim_{\eps\to0}\mathbb{E}\sup_{s\in[t,T]}|X^{t,x,\eps}_s-X^{t,x}_s|^2ds=0,
\end{eqnarray*}
and there exists a subsequence of $\{X^{t,x,\eps}_s\}$, still denoted by
$\{X^{t,x,\eps}_s\}$, which satisfies
\begin{eqnarray*}
\lim_{\eps\to0}\sup_{s\in[t,T]}|X^{t,x,\eps}_s-X^{t,x}_s|^2ds=0\ \ \ {\rm
a.s.}
\end{eqnarray*}
In the rest of arguments, we always consider this a.s. continuous
subsequence.

In order to get \eqref{dz29}, we need to deal with the following convergence:
\begin{eqnarray}\label{dz30}
  &&\lim_{\eps\to0}\mathbb{E}\int_t^T|c^\eps(s,X^{t,x,\eps}_s)Y^{t,x,\eps}_s-c(s,X^{t,x}_s)Y^{t,x}_s|^2ds\nonumber\\
  &&\leq\lim_{\eps\to0}2\mathbb{E}\int_t^T\Big(|c^\eps(s,X^{t,x,\eps}_s)|^2|Y^{t,x,\eps}_s-Y^{t,x}_s|^2+|c^\eps(s,X^{t,x,\eps}_s)-c(s,X^{t,x}_s)|^2|Y^{t,x}_s|^2\Big)ds\nonumber\\
   &&\leq\lim_{\eps\to0}2K_1^2\mathbb{E}\int_t^T|Y^{t,x,\eps}_s-Y^{t,x}_s|^2ds\nonumber\\
   &&\ \ \ \ +4\mathbb{E}\int_t^T\lim_{\eps\to0}\Big(\sup_{y\in\mathbb{R}^d}|c^\eps(s,y)-c(s,y)|^2+K_1^2|X^{t,x,\eps}_s-X^{t,x}_s|^2\Big)|Y^{t,x}_s|^2ds\nonumber\\
  &&=\lim_{\eps\to0}2K_1^2\mathbb{E}\int_t^T|Y^{t,x,\eps}_s-Y^{t,x}_s|^2ds.
\end{eqnarray}
Similarly, we have
\begin{eqnarray}\label{dz31}
\lim_{\eps\to0}\mathbb{E}\int_t^T|\nu^{\eps}(s,X^{t,x,\eps}_s)Z^{t,x,\eps}_s-\nu(s,X^{t,x}_s)Z^{t,x}_s|^2ds\leq\lim_{\eps\to0}2K_1^2\mathbb{E}\int_t^T|Z^{t,x,\eps}_s-Z^{t,x}_s|^2ds.\
\ \ \ \
\end{eqnarray}
By Sobolev embedding theorem again, it yields that
\begin{eqnarray}\label{dz32}
&&\lim_{\eps\to0}\mathbb{E}\int_t^T|F^{\eps}(s,X^{t,x,\eps}_s)-F(s,X^{t,x}_s)|^2ds\nonumber\\
&&\leq\lim_{\eps\to0}2\mathbb{E}\int_t^T\|F(s,\omega)\|_{W^{1,p}}|X^{t,x,\eps}_s-X^{t,x}_s|^\alpha ds\nonumber\\
&&\leq\lim_{\eps\to0}2\Big(\mathbb{E}\int_t^T\|F(s,\omega)\|^2_{W^{1,p}}ds\Big)^{1\over2}\Big(\mathbb{E}\int_t^T|X^{t,x,\eps}_s-X^{t,x}_s|^{2\alpha}ds\Big)^{1\over2}\nonumber\\
&&\leq
\lim_{\eps\to0}C\Big(\mathbb{E}\int_t^T|X^{t,x,\eps}_s-X^{t,x}_s|^2ds\Big)^{\alpha\over2}=0.
\end{eqnarray}
Similarly, we also obtain
\begin{eqnarray}\label{dz33}
\lim_{\eps\to0}\mathbb{E}|\varphi^{\eps}(X^{t,x,\eps}_T)-\varphi(X^{t,x}_T)|^2\leq
\lim_{\eps\to0}C\Big(\mathbb{E}|X^{t,x,\eps}_T-X^{t,x}_T|^2\Big)^{\alpha\over2}=0.
\end{eqnarray}

Then applying It$\hat {\rm o}$ formula to ${\rm
e}^{Ks}|Y^{t,x,\eps}_s-Y^{t,x}_s|^2$ for $K\in\mathbb{R}^1$, we have
\begin{eqnarray*}
&&\mathbb{E}{\rm e}^{Ks}|Y^{t,x,\eps}_s-Y^{t,x}_s|^2+K\mathbb{E}\int_s^T{\rm e}^{Kr}|Y^{t,x,\eps}_r-Y^{t,x}_r|^2dr+\mathbb{E}\int_s^T{\rm e}^{Kr}|Z^{t,x,\eps}_r-Z^{t,x}_r|^2dr\nonumber\\
&&\leq\mathbb{E}{\rm e}^{KT}|\varphi^{\eps}(X^{t,x,\eps}_T)-\varphi(X^{t,x}_T)|^2+(2+4K_1^2)\mathbb{E}\int_s^T{\rm e}^{Kr}|Y^{t,x,\eps}_r-Y^{t,x}_r|^2dr\nonumber\\
&&+\mathbb{E}\int_s^T{\rm e}^{Kr}|c^\eps(s,X^{t,x,\eps}_r)Y^{t,x,\eps}_r-c(s,X^{t,x}_r)Y^{t,x}_r|^2dr\nonumber\\
&&+{1\over{4K_1^2}}\mathbb{E}\int_s^T{\rm e}^{Kr}|\nu^{\eps}(s,X^{t,x,\eps}_r)Z^{t,x,\eps}_r-\nu(s,X^{t,x}_r)Z^{t,x}_r|^2dr\nonumber\\
&&+\mathbb{E}\int_s^T{\rm
e}^{Kr}|F^{\eps}(s,X^{t,x,\eps}_s)-F(s,X^{t,x}_s)|^2ds.
\end{eqnarray*}
Hence, using (\ref{dz30})-(\ref{dz33}) we have
\begin{eqnarray*}
\lim_{\eps\to0}(K-2-6K_1^2)\mathbb{E}\int_s^T{\rm
e}^{Kr}|Y^{t,x,\eps}_r-Y^{t,x}_r|^2dr+{1\over2}\mathbb{E}\int_s^T{\rm
e}^{Kr}|Z^{t,x,\eps}_r-Z^{t,x}_r|^2dr=0.
\end{eqnarray*}
Taking $K$ sufficiently large we immediately get
\begin{eqnarray*}
\lim_{\eps\to0}\mathbb{E}\int_t^T\Big(|Y^{t,x,\eps}_s-Y^{t,x}_s|^2+|Z^{t,x,\eps}_s-Z^{t,x}_s|^2\Big)ds=0.
\end{eqnarray*}
Then \eqref{dz29} follows by a standard method of B-D-G inequality.

\emph{Step 3.} On the other hand, we can further prove
\begin{eqnarray*}
 \mathbb{E}\int_{0}^{T}\|u^\eps(t)-u(t)\|_{0,2}^2dt\longrightarrow0.
\end{eqnarray*}
Indeed, similar to inequality \eqref{100}, it is not hard to prove
\begin{eqnarray*}
&&\bE\int_{0}^{T}{\rm e}^{\lambda t}\left\Vert u(t)^{\varepsilon
}-u(t)\right\Vert
_{0,2}^{2}dt\\
&&\leq C{\rm e}^{\lambda T}\bE\left\Vert \varphi ^{\varepsilon }-\varphi
\right\Vert _{0,2}^{2}+C\bE\int_{0}^{T}{\rm e}^{\lambda t}\left\Vert
F(t)^{\varepsilon
}-F(t)\right\Vert _{0,2}^{2}dt \\
&&\ \ \ +C\bE\int_{0}^{T}\int_{\R^{d}}{\rm e}^{\lambda t}(u^{\varepsilon
}-u)\Big[(\alpha ^{\varepsilon ,ij}-\alpha ^{ij})u_{x^{i}x^{j}}+(\sigma
^{\varepsilon
,ik}-\sigma ^{ik})q_{x^{i}}^{k}\\
&&\ \ \ \ \ \ \ \ \ \ \ \ \ \ \ \ \ \ \ \ \ \ \ \ \ \ \ \ \ \ \ \ \ \ \ \ \
+(b^{\varepsilon ,i}-b^{i})u_{x^{i}}+(c^{\varepsilon }-c)u+(\nu ^{\varepsilon
}-\nu )q\Big]dxdt,
\end{eqnarray*}%
where $\alpha ^{\varepsilon ,ij}=\frac{1}{2}\sigma ^{\varepsilon ,ik}\sigma
^{\varepsilon ,jk}$, $\alpha ^{ij}=\frac{1}{2}\sigma ^{ik}\sigma ^{jk}$ and $%
\lambda $ is a sufficiently large number. By condition $(\A_1)$ and constructions of smootherized coefficients, we can deduce for each $\left( t,x,\omega \right)$,%
\begin{eqnarray*}
&&\left\vert D(\alpha ^{\varepsilon ,ij}-\alpha ^{ij})\right\vert +\left\vert
D(\sigma ^{\varepsilon ,ik}-\sigma ^{ik})\right\vert +\left\vert \alpha
^{\varepsilon ,ij}-\alpha ^{ij}\right\vert +\left\vert \sigma
^{\varepsilon ,ik}-\sigma ^{ik}\right\vert  \\
&&+\left\vert b^{\varepsilon ,i}-b^{i}\right\vert +\left\vert c^{\varepsilon
}-c\right\vert +\left\vert \nu ^{\varepsilon }-\nu \right\vert  \leq
C\varepsilon.
\end{eqnarray*}%
Thus, by integration by parts, it turns out that
\begin{eqnarray*}
&&\lim_{\varepsilon \rightarrow 0}\bE\int_{0}^{T}{\rm e}^{\lambda
t}\left\Vert
u^{\varepsilon}(t)-u(t)\right\Vert _{0,2}^{2}dt \\
&&\leq\lim_{\varepsilon \rightarrow 0}\varepsilon C{\rm e}^{\lambda T}
\bE\int_{0}^{T}\int_{\R^{d}}\Big[ |D\left( u^{\varepsilon }-u\right) |\left(
\left\vert Du\right\vert +\left\vert q\right\vert \right) +|u^{\varepsilon
}-u|\left( \left\vert Du\right\vert +\left\vert u\right\vert +\left\vert
q\right\vert \right)\Big] dxdt \\
&&\leq\lim_{\varepsilon \rightarrow 0}\varepsilon C{\rm e}^{\lambda T}
\bE\int_{0}^{T}\left( \left\Vert u^{\varepsilon }-u\right\Vert
_{1,2}^{2}+\left\Vert u\right\Vert _{1,2}^{2}+\left\Vert q\right\Vert
_{0,2}^{2}\right) dt \\
&&\leq\lim_{\varepsilon \rightarrow 0}\varepsilon C{\rm e}^{\lambda
T}\bE\Big[ \left\Vert \varphi ^{\varepsilon }\right\Vert
_{1,2}^{2}+\left\Vert \varphi \right\Vert _{1,2}^{2}+\int_{0}^{T}\left(
\left\Vert F^{\varepsilon}(t)\right\Vert _{1,2}^{2}+\left\Vert
F(t)\right\Vert _{1,2}^{2}\right) dt\Big]=0.
\end{eqnarray*}

Therefore, there exists a subsequence of $\{u^\eps\}$, still denoted by
$\{u^\eps\}$, such that $u^\eps(t,x)\longrightarrow u(t,x)$ as $\eps\to0$ for
a.e. $t\in[0,T]$, $x\in\R^d$ a.s., which implies that $Y^{t,x}_t=u(t,x)$ for
a.e. $t\in[0,T]$, $x\in\R^d$ a.s. in view of \eqref{dz29}. Noticing Corollary
\ref{dz111}, we know that $u(t,x)$ is continuous with respect to $(t,x)$,
which together with Proposition \ref{26} leads to
 \begin{eqnarray*}
  u(t,x)=Y^{t,x}_t
  \ \ {\rm for}\ {\rm all}\ t\in[0,T],\ x\in\R^d\ {\rm a.s.}
  \end{eqnarray*}
In particular,
 \begin{eqnarray*}
  u(s,X^{t,x}_s)=Y^{s,X^{t,x}_s}_s
  \ \ {\rm for}\ {\rm all}\ s\in[t,T],\ x\in\R^d\ {\rm a.s.}
  \end{eqnarray*}
By the uniqueness of solution of FBSDE \eqref{dz21}, \eqref{dz19} follows.
\end{proof}

Utilizing the connection between FBSDE and BSPDE in the linear case, we
further study the same kind of connection in the semi-linear case.
\begin{thm}\label{28}
  Suppose that the conditions in Theorem \ref{c4:bounded} are satisfied, then
  we have a same kind of connection as \eqref{dz19} between the solution $u$ to BSPDE \eqref{eq:bspde} and the solution $Y$ to FBSDE \eqref{eq:bsde}.
\end{thm}
\begin{proof}
Let $u$ be the solution of BSPDE \eqref{eq:bspde} and
$\hat{F}(t,x)=f(t,x,u(t,x))$. Obviously, $\hat{F}(t,x)\in L^{p}_{\sP}W^{1,p}$
and we regard BSPDE \eqref{eq:bspde} as a linear equation with generator
$\hat{F}$. By Theorem \ref{27} we know
\begin{eqnarray*}
  u(s,X^{t,x}_s)=\hat{Y}^{t,x}_s
  \ \ {\rm for}\ {\rm all}\ s\in[t,T],\ x\in\R^d\ {\rm a.s.},
  \end{eqnarray*}
where $(\hat{Y}^{t,x}_s,\hat{Z}^{t,x}_s)_{s\in[t,T]}$ is the solution of
FBSDE with generator $\hat{F}$ as follows:
\begin{eqnarray*}\label{dz27}
  \hat{Y}^{t,x}_s = \vf(X^{t,x}_T) + \int_{s}^{T} \hat{F}(r,X^{t,x}_r) dr - \int_{s}^{T} \hat{Z}^{t,x}_r dW_r.
\end{eqnarray*}
By the definition of $\hat{F}$, we know that
$(\hat{Y}^{t,x}_s,\hat{Z}^{t,x}_s)$ is also the solution of FBSDE
\eqref{eq:bsde}. Since Remark \ref{29}, the solution of FBSDE \eqref{eq:bsde}
is unique. Then Theorem \ref{28} follows immediately.
\end{proof}

\bigskip

\noindent\textbf{References}

\bibliographystyle{model1b-num-names}

\end{document}